\documentclass[12pt]{amsart}%
\usepackage{color}
\usepackage[right=2.7cm,left=2.7cm,top=3.8cm,bottom=3.5cm]{geometry}

\usepackage{float}
\usepackage[pdftex]{graphicx}
\usepackage{amsmath}
\usepackage{amsfonts}
\usepackage{bbold}
\usepackage{dsfont}
\usepackage{amssymb}
\usepackage{flafter}
\usepackage{float}
\usepackage{amsfonts}
\usepackage[pdftex]{graphicx}
\usepackage{amsmath}
\usepackage{amsfonts}
\usepackage{dsfont}
\usepackage{amssymb}
\usepackage{flafter}
\usepackage[pdffitwindow=false, pdfview=FitH, pdfstartview=FitH,
plainpages=false]{hyperref}%
\providecommand{\U}[1]{\protect \rule{.1in}{.1in}}
\newtheorem{theorem}{Theorem}[section]
\newtheorem{lemma}[theorem]{Lemma}
\newtheorem{proposition}[theorem]{Proposition}
\theoremstyle{definition}
\newtheorem{definition}[theorem]{Definition}

\theoremstyle{remark}
\newtheorem{remark}[theorem]{Remark}

\numberwithin{equation}{section}
\newcommand{\iii}{{\, \vert\kern-0.25ex\vert\kern-0.25ex\vert\, }}

\begin{document}
\thispagestyle{empty}
\title[stabilization of a locally transmission problems of two wave systems]{stabilization of a locally transmission problems of two strongly-weakly coupled wave systems }

\author{Wafa Ahmedi}
\address{ UR Analysis and Control of PDE's, UR 13ES64\\ Higher School of Sciences and Technology of Hammam Sousse, University of Sousse, Tunisia.}
\email{ahmadisouad53@gmail.com}

\author{Akram Ben Aissa}
\address{ Lab Analysis and Control of PDEs, LR22ES03\\ Higher Institute of transport and Logistics of Sousse, University of Sousse,  Tunisia.}
\email{akram.benaissa@fsm.rnu.tn}

\date{}
\begin{abstract}
In this paper, we embark on a captivating exploration of the stabilization of locally transmitted problems within the realm of two interconnected wave systems. To begin, we wield the formidable Arendt-Batty criteria\cite{AW} to affirm the resolute stability of our system. Then, with an artful fusion of a frequency domain approach and the multiplier method, we unveil the exquisite phenomenon of exponential stability, a phenomenon that manifests when the waves of the second system synchronize their propagation speeds. In cases where these speeds diverge, our investigation reveals a graceful decay of our system's energy, elegantly characterized by a polynomial decline at a rate of $t^{-1}$.

\end{abstract}


\maketitle

\section{Introduction}
    When a vibrating source disturbs the first particular of a medium, a wave is created. This phenomena begins to travel from particle to particle along the medium, which is typically modelled by a wave equation. In order to suppress those vibrations, the most common approach is adding damping.\\
    Over the past few years, the stabilization of wave systems (simple or coupled) with localized damping has attracted the attention of many authors, in the one dimensional case: (see \cite{Lk}, \cite{Lz}, \cite{hs}, \cite{AHh}, \cite{A}, \cite{Aa}, \cite{IB}, \cite{AG} ), see also (\cite{Abn}, \cite{AG} ). In this two papers, it was proved that the smoothness of the damping and coupling coefficients play a crucial role in the stability and regularity of the solution of the studied system. Especially, in \cite{Abn}, Akil, Badawi and Nicaise studied the stability of locally coupled wave equations with local Kelvin-Voigt damping when the supports of damping and coupling coefficients are disjoint. They proved that the energy of thier system decays polynomially with rates $t^{-1}$ and $t^{-4}$ in the differents cases.
    On the other hand, there are many publications in the multi-dimensional setting (see \cite{HAS}, \cite{Awy}, \cite{Aak},\cite{Aaw}, \cite{Te2}, \cite{NN}, \cite{Naj}, \cite{Nass}, \cite{lab1}, \cite{lab2}, \cite{AA}, \cite{AAI}, \cite{Amm}, \cite{exp}, \cite{f1}, \cite{f2}, \cite{f3}). Our purpose in this work is to study a more general problem. But, before stating our main contributions, let us start by
    recalling some previous results for such systems. In 2020, S.Gerbi et al. in \cite{SCA} proved the exponential decay rate of solution when the waves propagate with equal speeds, the coupling region is a subset of the damping region and satisfies the geometric control condition GCC and that the damping and the coupling coefficients are in $W^{1,\infty}(\Omega)$. In the same direction, recently, Wehbe, Ibtissam and Akil in \cite{AIW}, showed that the energy of the smooth solutions of the system decays polynomially of type $t^{-1}$, by considering that both the damping and the coupling coefficients are non smooth. Then, in \cite{AAI}, they  generalize this work to a multidimensional case and they study the stability of the system under several
    geometric control conditions. they establish polynomial stability when there is an intersection between the damping and the coupling regions. Also, when the coupling region is a subset of the damping region and under Geometric Control Condition GCC.\\
     
    Within the intricate weave of this paper, our focus converges on an intriguing question: What enigmatic qualities define the stability of our transmission problems (\ref{pb})?. Indeed, this problem involves two wave systems: the first one is weakly coupled and the second system is strongly coupled with non smooth coefficients.
    To the best of our knowledge, it seems that no result in the literature exists concerning our problem (\ref{pb}), especially in the one dimensional case. The goal of the present paper is to fill this gap by studying the stability of the following locally transmission problem:
\begin{equation}\label{pb}
\left \{
\begin{aligned}
&\displaystyle u_{tt}- a_{1}  u_{xx}+d_{1}(x) u_{t} +c_{1}(x)y=0, &(x,t) \in &\; (0,L_{0}) \times \mathbb{R}^{*}_{+},\\
&\displaystyle y_{tt}- y_{xx}+c_{1}(x)u=0, &(x,t) \in &\; (0,L_{0}) \times \mathbb{R}^{*}_{+},\\
&\displaystyle \varphi_{tt}- a_{2}  \varphi_{xx}+d_{2}(x) \varphi_{t} +c_{2}(x) \psi_{t}=0, &(x,t) \in &\; (L_{0},L) \times \mathbb{R}^{*}_{+},\\
&\displaystyle \psi_{tt}- \psi_{xx}-c_{2}(x) \varphi_{t}=0, &(x,t) \in &\; (L_{0},L) \times \mathbb{R}^{*}_{+},
\end{aligned} \right.
\end{equation}
with fully Dirichlet boundary conditions,
\begin{equation}\label{bord}
u(0,t)=y(0,t)=\varphi(L,t)=\psi(L,t)=0, \;\; \;\; t \in   \mathbb{R}^{*}_{+},
\end{equation}
and the following transmission conditions,
\begin{equation}\label{tran}
\left \{
\begin{aligned}
&u(L_{0},t)=\varphi(L_{0},t),\;y(L_{0},t)=\psi(L_{0},t), &t \in &\;  \mathbb{R}^{*}_{+},\\
&a_{1}u_{x}(L_{0},t)=a_{2} \varphi_{x}(L_{0},t),\; y_{x}(L_{0},t)=\psi_{x}(L_{0},t), &t \in &\;  \mathbb{R}^{*}_{+},
\end{aligned} \right.
\end{equation}
and with the following initial data
\begin{equation}\label{init}
(u,y,\varphi,\psi,u_{t},y_{t},\varphi_{t},\psi_{t})(x,0)=(u_{0},y_{0},\varphi_{0},\psi_{0},u_{1},y_{1},\varphi_{1},\psi_{1}).
\end{equation}
where
\begin{equation}\label{d1}
d_{1}(x)=\left \{
\begin{aligned}
&d_{1} & \text{if}  \;\; x \in (\alpha_{2},\alpha_{4}) \\
&0 &\text{otherwise}
\end{aligned} \right. \;\; \;\; \;\; \;\; \;\;
c_{1}(x)=\left \{
\begin{aligned}
&c_{1} & \text{if}  \;\; x \in (\alpha_{1},\alpha_{3}) \\
&0 &\text{otherwise},
\end{aligned} \right.
\end{equation}
\begin{equation}\label{d2}
d_{2}(x)=\left \{
\begin{aligned}
&d_{2} & \text{if}  \;\; x \in (\beta_{2},\beta_{4}) \\
&0 &\text{otherwise}
\end{aligned} \right. \;\; \;\; \text{and} \;\; \;\;
c_{2}(x)=\left \{
\begin{aligned}
&c_{2} & \text{if}  \;\; x \in (\beta_{1},\beta_{3}) \\
&0 &\text{otherwise},
\end{aligned} \right.
\end{equation}
and $a_{1}, a_{2}, d_{1}, d_{2}$ are strictly positives constants, $c_{1}, c_{2} \in \mathbb{R^{*}}$, with
\begin{equation}\label{infty}
\left|  c_{1} \right| <\frac{1}{C_{0}},
\end{equation}
where $C_{0}$ denotes the Poincar$\acute{e}$ constant. More precisely, $C_{0}$ is the smallest positive constant such that
\begin{equation}
\int_{0}^{L_{0}}\left| f\right|^{2} dx \leq C_{0} \int_{0}^{L_{0}}\left| f_{x}\right|^{2} dx, \;\; \;\; \forall f \in \mathnormal{H}_{0}^{1}(0,L_{0}).
\end{equation}
We consider $0<\alpha_{1}<\alpha_{2}<\alpha_{3}<\alpha_{4}<L_{0}<\beta_{1}<\beta_{2}<\beta_{3}<\beta_{4}<L$. $\left( \text{See Figure}\; \ref{fig:1}\right) $ \\

The paper is structured as follows: First in Sect.\ref{sect}, we prove the well-posedness of our system by using semigroup approach. Then in Sect.\ref{sec3}, following a general criteria of Arendt and Batty, we show the strong stability of our problem. Next, in Sect\ref{exp}, by using the frequency domain approach combining with a specific multiplier method, we establish exponential stability of the solution if and only if the waves of the second coupled equations have the same speed of propagation (i.e., $a_{2} = 1$). In the case when $a_{2} \neq 1$, we prove that the energy of our problem decays polynomially with the rate $t^{-1}$. 
\begin{figure}
	\centering
	\includegraphics[width=0.7\textwidth]{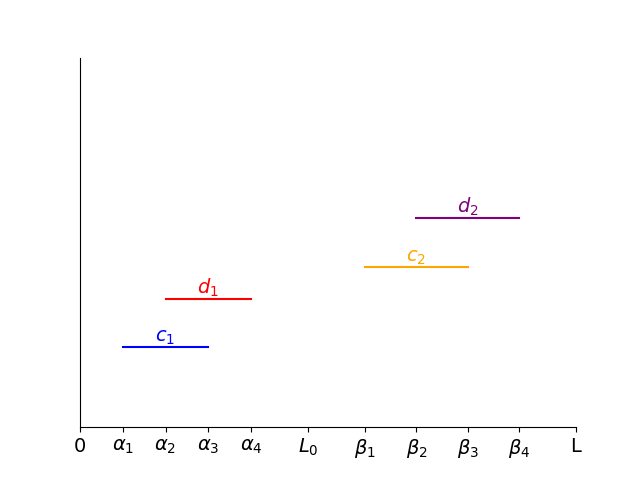}
	\caption{Geometric description of the functions d$_{1}$, d$_{2}$, c$_{1}$ and c$_{2}$ }
	\label{fig:1}
\end{figure}
\newpage
\section{Well-posedness} \label{sect}
This section is devoted to establish the well-posedness of the system (\ref{pb})-(\ref{init}) using a semi-group approach.\\
\textbf{Case1.} If $d_{1}(x)=0\; \text{in} \; (0,L_{0})$.\\
Let $(u,u_{t},y,y_{t},\varphi,\varphi_{t},\psi,\psi_{t})$ be a regular solution of the system (\ref{pb})-(\ref{init}). The energy of the system is given by
\begin{align}
E(t)&=\frac{1}{2} \int_{0}^{L_{0}}\left(  
| u_{t}|^{2}+a_{1}| u_{x}|^{2}+|y_{t}|^{2}+|y_{x}|^{2}+2 \Re \left( c_{1}(x)u\overline{y}\right) \right)dx\\ \nonumber
&+ \frac{1}{2} \int_{L_{0}}^{L}\left(  
| \varphi_{t}|^{2}+a_{2}| \varphi_{x}|^{2}+|\psi_{t}|^{2}+|\psi_{x}|^{2}\right)dx .
\end{align}
A straightforward computation gives
\begin{equation}
E'(t)=- \int_{L_{0}}^{L} d_{2}(x) |\varphi_{t}|^{2} dx \leq 0.
\end{equation}
Thus, the system (\ref{pb})-(\ref{init}) is dissipative in the sense that its energy is a non increasing function with respect to the time variable $t$. We introduce the following Hilbert spaces
$$
\mathnormal{H}^{1}_{L}(a,b)=\left\lbrace f \in \mathnormal{H}^{1}(a,b);\; f(a)=0\right\rbrace ,\;\; \;\; \mathnormal{H}^{1}_{R}(a,b)=\left\lbrace f \in \mathnormal{H}^{1}(a,b);\; f(b)=0\right\rbrace,
$$ 
for any real numbers $a, b$ such that $a < b$. The energy space $\mathcal{H}$ is now defined by 
\begin{align}
\mathcal{H}=  \Big \{ &
\left[  \mathnormal{H}_{L}^{1}(0,L_{0}) \times \mathnormal{L}^{2}(0,L_{0})\right]^{2} \times \left[  \mathnormal{H}_{R}^{1}(L_{0},L) \times \mathnormal{L}^{2}(L_{0},L)\right]^{2};\; u(L_{0})=\varphi(L_{0}) \\ \nonumber
& \text{and} \;\; \;\; y(L_{0})=\psi(L_{0})
\Big \}
\end{align}
equipped with the following norm
\begin{align*}
\|U\|^{2}_{\mathcal{H}}=&a_{1} \| u_{x}\|^{2}_{\mathnormal{L}^{2}(0,L_{0})}+\| v\|^{2}_{\mathnormal{L}^{2}(0,L_{0})}+\| y_{x}\|^{2}_{\mathnormal{L}^{2}(0,L_{0})}+\|z\|^{2}_{\mathnormal{L}^{2}(0,L_{0})} + 2 \Re \int_{0}^{L_{0}} c_{1}(x)u\overline{y} dx\\
&+ a_{2} \| \varphi_{x}\|^{2}_{\mathnormal{L}^{2}(L_{0},L)}
+\| \eta\|^{2}_{\mathnormal{L}^{2}(L_{0},L)}+\| \psi_{x}\|^{2}_{\mathnormal{L}^{2}(L_{0},L)}+\|\xi\|^{2}_{\mathnormal{L}^{2}(L_{0},L)}, 
\end{align*}
for all $U=(u,v,y,z,\varphi,\eta,\psi,\xi)^{\mathsf{T}} \in \mathcal{H} $.\\
Let $U=(u,u_{t},y,y_{t},\varphi,\varphi_{t},\psi,\psi_{t})^{\mathsf{T}}$, it's easy to see that problem (\ref{pb})-(\ref{init}) is formally equivalent to the following abstract evolution equation in the Hilbert space $\mathcal{H}$
\begin{equation}\label{koko}
\begin{cases}
U'(t) = \mathcal{A}U(t), \;\; t>0
\\
U(0)=U_{0}=(u_{0},u_{1},y_{0},y_{1},\varphi_{0},\varphi_{1},\psi_{0},\psi_{1})^{\mathsf{T}}, 
\end{cases} 
\end{equation}
and the unbounded operator $\mathcal{A}$ is defined by
\begin{equation}\label{eeq}
\mathcal{A}U=  \begin{pmatrix} v \\ 
 a_{1}  u_{xx}-c_{1}(x)y \\ z \\ y_{xx}-c_{1}(x)u\\ \eta \\ 
 a_{2}  \varphi_{xx}-c_{2}(x)\xi-d_{2}(x) \eta \\ \xi \\ \psi_{xx}+c_{2}(x) \eta  \end{pmatrix} ,
\end{equation}
for all $U=(u,v,y,z,\varphi,\eta,\psi,\xi)^{\mathsf{T}} \in \mathnormal{D}(\mathcal{A})$, with domain
\begin{align*}
\mathnormal{D}(\mathcal{A})=& \Big \{ U \in \mathcal{H} ;\; v,z  \in  \mathnormal{H}^{1}_{L}(0,L_{0}),u,y \in   \mathnormal{H}^{2}(0,L_{0})\cap \mathnormal{H}^{1}_{L}(0,L_{0}), \eta, \xi  \in  \mathnormal{H}^{1}_{R}(L_{0},L),\\
&\varphi, \psi \in   \mathnormal{H}^{2}(L_{0},L)\cap \mathnormal{H}^{1}_{R}(L_{0},L) \;\text{and}\;
  u(L_{0})=\varphi(L_{0}) , y(L_{0})=\psi(L_{0})  \Big \} .
\end{align*}
\begin{proposition}\label{po}
The unbounded linear operator $\mathcal{A}$ generates a $C_{0}$-semigroup of contractions on $\mathcal{H}$.
\end{proposition}
\begin{proof}
Using Lumer-Phillips theorem (see\cite{PA}), it is sufficient to prove that $\mathcal{A}$ is a maximal
dissipative operator so that $\mathcal{A}$ generates a C$_{0}$-semigroup of contractions on $\mathcal{H}$. First, let $U=(u,v,y,z,\varphi,\eta,\psi,\xi)^{\mathsf{T}} \in \mathnormal{D}(\mathcal{A}) $. Then, integrating by parts we have
\begin{equation}\label{re}
\Re \left\langle  \mathcal{A}U,U  \right\rangle _{\mathcal{H}} =-\int_{L_{0}}^{L}  d_{2}(x) \left| \eta\right| ^{2} dx\leq 0.
\end{equation}
This implies that $\mathcal{A}$ is dissipative. Now, let us go on with maximality. Let\\ $F=(f_{1},f_{2},f_{3},f_{4},f_{5},f_{6},f_{7},f_{8})^{\mathsf{T}} \in \mathcal{H} $, we look for $U=(u,v,y,z,\varphi,\eta,\psi,\xi)^{\mathsf{T}} \in \mathnormal{D}(\mathcal{A}) $ solution of
\begin{equation}\label{inv}
-\mathcal{A}U=F.
\end{equation}
Equivalently, we have the following system
\begin{align} \label{1}
-v&=f_{1},\\ \label{2}
-a_{1} u_{xx}+c_{1}(x) y&=f_{2},\\ \label{3}
-z&=f_{3}, \\ \label{4}  
- y_{xx}+c_{1}(x)u  &=f_{4},\\ \label{5}
-\eta&=f_{5},\\ \label{6}
-a_{2} \varphi_{xx}+d_{2}(x) \eta+c_{2}(x)\xi
&=f_{6},\\ \label{7}
-\xi&=f_{7}, \\ \label{8}  
- \psi_{xx}-c_{2}(x)\eta  &=f_{8}.
\end{align}
Inserting (\ref{5}) and (\ref{7}) into (\ref{6}) and (\ref{8}), we get
\begin{equation}\label{*}
-a_{2} \varphi_{xx} =f_{6}+c_{2}(x) f_{7}+ d_{2}(x)f_{5},
\end{equation}
and
\begin{equation}\label{**}
-\psi_{xx} =f_{8}-c_{2}(x) f_{5}.
\end{equation}
Multiplying (\ref{2}), (\ref{*}) by $\overline{\varPhi_{1}} \in \mathnormal{H}^{1}_{0}(0,L)$ and  (\ref{4}), (\ref{**}) by $\overline{\varPhi_{2}} \in \mathnormal{H}^{1}_{0}(0,L)$, integrating over $(0,L)$, we get
\begin{equation}\label{eq1}
-a_{1} \overline{\varPhi_{1}}(L_{0}) u_{x}(L_{0}) + a_{1} \int_{0}^{L_{0}}  u_{x}  \left( \overline{\varPhi_{1}}\right) _{x} dx +\int_{0}^{L_{0}} c_{1}(x)y   \overline{\varPhi_{1}} dx= \int_{0}^{L_{0}}f_{2} \overline{\varPhi_{1}} dx, 
\end{equation}
\begin{equation}\label{eq2}
- \overline{\varPhi_{2}}(L_{0}) y_{x}(L_{0}) +  \int_{0}^{L_{0}}  y_{x}  \left( \overline{\varPhi_{2}}\right)_{x} dx +\int_{0}^{L_{0}} c_{1}(x)u   \overline{\varPhi_{2}} dx= \int_{0}^{L_{0}}f_{4} \overline{\varPhi_{2}} dx, 
\end{equation}
\begin{equation}\label{eq3}
a_{2} \overline{\varPhi_{1}}(L_{0}) \varphi_{x}(L_{0}) + a_{2} \int_{L_{0}}^{L}  \varphi_{x}  \left( \overline{\varPhi_{1}}\right) _{x} dx =\int_{L_{0}}^{L}\left( f_{6}+ c_{2}(x)f_{7}+d_{2}(x)f_{5}\right)    \overline{\varPhi_{1}} dx,
\end{equation}
and
\begin{equation}\label{eq4}
 \overline{\varPhi_{2}}(L_{0}) \psi_{x}(L_{0}) +  \int_{L_{0}}^{L}  \psi_{x}  \left( \overline{\varPhi_{2}}\right) _{x} dx =\int_{L_{0}}^{L}\left( f_{8}- c_{2}(x)f_{5} \right)    \overline{\varPhi_{2}} dx.
\end{equation}
Adding the above equations, we obtain
the following variational problem:
\begin{equation}\label{fo}
\Lambda\left( \left( \left( u,\varphi \right) ,\left( y, \psi \right) \right), \left( \varPhi_{1},\varPhi_{2}\right)\right)  =l(\varPhi_{1},\varPhi_{2}), \;\; \;\; \forall (\varPhi_{1},\varPhi_{2}) \in \mathnormal{H}^{1}_{0}(0,L) \times \mathnormal{H}^{1}_{0}(0,L),
\end{equation} 
where
$$
\Lambda \left( \left( \left( u, \varphi \right) ,\left( y, \psi \right) \right) ,\left( \varPhi_{1},\varPhi_{2}\right)\right)  =\vartheta_{1}\left( \left( u,y \right) ,  \left( \varPhi_{1},\varPhi_{2}\right) \right)+\vartheta_{2}\left( \left( \varphi, \psi \right) , \left( \varPhi_{1},\varPhi_{2}\right) \right),
$$
with
\begin{equation*}
\left \{
\begin{aligned}
&\vartheta_{1}\left( \left( u,y\right) ,  \left( \varPhi_{1},\varPhi_{2}\right) \right)=\int_{0}^{L_{0}} \left( a_{1}  u_{x} ( \overline{\varPhi_{1}})_{x} +y_{x} ( \overline{\varPhi_{2}})_{x}   +c_{1}(x) y \overline{\varPhi_{1}} +c_{1}(x) u \overline{\varPhi_{2}} \right) dx ,\\
&\vartheta_{2}\left( \left( \varphi,\psi\right), \left( \varPhi_{1},\varPhi_{2}\right) \right)= a_{2}\int_{L_{0}}^{L}   \varphi_{x} ( \overline{\varPhi_{1}})_{x} dx + \int_{L_{0}}^{L} \psi_{x} ( \overline{\varPhi_{2}})_{x} dx ,\\
\end{aligned} \right.
\end{equation*}
and 
\begin{align*}
l(\varPhi_{1},\varPhi_{2})= &\int_{0}^{L_{0}} \left( f_{2} \overline{\varPhi_{1}}+ f_{4} \overline{\varPhi_{2}} \right)  dx +  \int_{L_{0}}^{L} \left( f_{6} +c_{2}(x)f_{7}+d_{2}(x) f_{5}\right) \overline{\varPhi_{1}} dx \\ 
&+ \int_{L_{0}}^{L}\left(  f_{8}-c_{2}(x)f_{5}\right)  \overline{\varPhi_{2}}  dx .
\end{align*}
First, thanks to (\ref{infty}) we have that $\vartheta_{1}$ is a bilinear, continuous and coercive form on $\left( \mathnormal{H}^{1}_{L}(0,L_{0}) \times \mathnormal{H}^{1}_{L}(0,L_{0}) \right) ^{2}$. Second it's easy to see that $\vartheta_{2}$ is a bilinear, continuous and coercive form on $\left( \mathnormal{H}^{1}_{R}(L_{0},L) \times \mathnormal{H}^{1}_{R}(L_{0},L) \right) ^{2}$ and $l$ is linear continuous form on $\mathnormal{H}^{1}_{0}(L_{0}, L) \times \mathnormal{H}^{1}_{0}(L_{0},L)$. Then, using Lax-Milgram theorem, we deduce that there exists $((u,\varphi),(y, \psi)) \in \mathnormal{H}^{1}_{0}(0,L) \times \mathnormal{H}^{1}_{0}(0,L)$ unique
solution of the variational problem (\ref{fo}). By using the classical elliptic regularity we deduce
that $u ,y \in \mathnormal{H}^{2}(0,L_{0})\cap  \mathnormal{H}^{1}_{L}(0,L_{0})$ and $\varphi, \psi \in \mathnormal{H}^{2}(L_{0},L)\cap  \mathnormal{H}^{1}_{R}(L_{0},L)$. Next by setting $v = -f_{1}$, $z = -f_{3}$, $\eta = -f_{5}$ and $\xi = -f_{7}$, we deduce that 
$U = (u, v, y, z,\varphi, \eta, \psi, \xi)^{\mathsf{T}} \in \mathnormal{D}(\mathcal{A})$ is solution of (\ref{inv}). To conclude, we need to show the uniqueness of such a solution. So, let $U = (u, v, y, z,\varphi, \eta, \psi, \xi)^{\mathsf{T}} \in \mathnormal{D}(\mathcal{A})$ be a solution of (\ref{inv}) with 
$$f_{1} = f_{2} = f_{3} = f_{4} =f_{5}=f_{6}=f_{7}=f_{8}= 0.
$$
Then we directly deduce that $v = z =\eta = \xi= 0$ and therefore $(u, y)\in \left[ \mathnormal{H}^{1}_{L}(0,L_{0}) \right] ^{2}$  satisfies (\ref{fo}) with $l(\varPhi_{1},\varPhi_{2})=0$ and $(\varphi, \psi)\in \left[ \mathnormal{H}^{1}_{R}(L_{0},L) \right] ^{2}$  satisfies (\ref{fo}) with $l(\varPhi_{1},\varPhi_{2})=0$. As $\vartheta_{1}$, $\vartheta_{2}$ are two sesquilinears, continuous coercive forms, we deduce that 
$$
u = y = \varphi = \psi = 0.
$$
In other words, $\ker\left( \mathcal{A}\right) =\left\lbrace 0\right\rbrace $. Consequently, we get $U = (u,- f_{1}, y,- f_{3},\varphi,- f_{5}, \psi,- f_{7})^{\mathsf{T}} \in \mathnormal{D}(\mathcal{A}) $ is  a unique solution of (\ref{inv}).\\

 Since $0\in \rho (\mathcal{A})$ the resolvent set of $\mathcal{A}$ we easily get  $\mathnormal{R} \left( \lambda I - \mathcal{A}\right) = \mathcal{H}$ for a sufficiently small $\lambda > 0$ (see \cite[Thm.1.2.4]{ZZ}). This, together
with the dissipativeness of $\mathcal{A}$, imply that $\mathnormal{D}(\mathcal{A})$ is dense in $\mathcal{H}$ (see\cite[Thm.4.6]{PA}). Then $\mathcal{A}$ is m-dissipative in
$\mathcal{H}$.
\end{proof}
As $\mathcal{A}$ generates a $C_{0}$-semigroup of contractions $\left(e^{t\mathcal{A}} \right)_{t\geq 0}$, we have the following result:
\begin{theorem} $\left( \text{Existence and uniqueness of the solution}\right) $.
	\begin{enumerate}
		\item If $ U_{0} \in \mathnormal{D}(\mathcal{A})$, then problem (\ref{koko}) admits a unique strong solution $U $ satisfying:
		$$
		U\in C^{1}\left( \mathbb{R}_{+},\mathcal{H}\right) \cap C^{0}\left( \mathbb{R}_{+},\mathnormal{D}(\mathcal{A})\right)  .
		$$
		\item If $ U_{0}  \in \mathcal{H}$, then problem (\ref{koko}) admits a unique weak solution $U$ satisfying:
		$$
		U\in C^{0}\left( \mathbb{R}_{+},\mathcal{H}\right) .
		$$
		\end{enumerate}
\end{theorem}
\section{Strong stability} \label{sec3}
Now the following result is about the strong stability of system (\ref{pb})-(\ref{init})
\begin{theorem}\label{pap}
The $C_{0}$-semigroup of contractions $\left(e^{t\mathcal{A}} \right)_{t\geq 0}$ is strongly stable in the energy space $\mathcal{H}$ in the sense that
$$
\lim_{t\to +\infty}\|e^{t\mathcal{A}}U_{0}\|_{\mathcal{H}}=0,\quad
\forall U_{0}\in \mathcal{H} .
$$
\end{theorem}
\begin{proof}
Since the resolvent of $\mathcal{A}$ is compact in $\mathcal{H}$, it follows from the Arendt-Batty's theorem (see\cite{AW}) that the system (\ref{pb})-(\ref{init}) is strongly stable if and only if $\mathcal{A}$ does not have pure imaginary eigenvalues, i.e. $\sigma(A) \cap \mathrm{i} \mathbb{R}=\emptyset$.
From Proposition\ref{po}, we have that $0 \in \rho (\mathcal{A}) $. Therefore, only $\sigma(A) \cap \mathrm{i} \mathbb{R^{*}}=\emptyset$ must be proved. For this purpose, suppose that there exists a real number $\lambda \neq 0$ and $U = (u, v, y, z,\varphi, \eta, \psi, \xi)^{\mathsf{T}} \in \mathnormal{D}(\mathcal{A})$ such that 
\begin{equation}\label{eg}
\mathcal{A}U = \mathrm{i}\lambda U.	
\end{equation}	
From (\ref{re}) and (\ref{eg}), we have
\begin{equation}\label{egg}
0=\Re\left( \mathrm{i}\lambda \|U\|^{2}_{\mathcal{H}}\right)=\Re \left( \langle \mathcal{A}U,U \rangle _{\mathcal{H}}\right)=-\int_{L_{0}}^{L}  d_{2}(x) \left|  \eta \right|^{2} dx.	
\end{equation}
Condition  (\ref{d2}) implies that
\begin{equation}\label{t}
\sqrt{d_{2}} \eta =0\; \text{in}\; (L_{0},L)  \;\; \;\; \text{and} \;\; \;\;  \eta =0  \; \text{in} \;  (\beta_{2},\beta_{4}).
\end{equation}
Detailing (\ref{eg}) and using (\ref{t}), we get the following equations
\begin{equation}\label{moo}
v=\mathrm{i}\lambda u\; \text{in} \; (0,L_{0}),\;\; \;\; z=\mathrm{i}\lambda y\; \text{in} \; (0,L_{0}),\;\; \;\; \eta=\mathrm{i}\lambda \varphi\; \text{in} \; (L_{0},L),\;\; \;\; \xi=\mathrm{i}\lambda \psi\; \text{in} \; (L_{0},L),
\end{equation}
and
\begin{equation}\label{mu}
\left \{
\begin{aligned}
&\displaystyle \lambda^{2}u+ a_{1}  u_{xx}-c_{1}(x)y=0, &x \in &\; (0,L_{0}) ,\\
&\displaystyle \lambda^{2}y+   y_{xx}-c_{1}(x)u=0, &x \in &\; (0,L_{0}) ,\\
&\displaystyle \lambda^{2}\varphi+ a_{2}  \varphi_{xx}-\mathrm{i} \lambda c_{2}(x)\psi=0, &x \in &\; (L_{0},L) ,\\
&\displaystyle \lambda^{2} \psi+  \psi_{xx} +\mathrm{i} \lambda c_{2}(x)\varphi=0, &x \in &\; (L_{0},L) .
\end{aligned} \right.
\end{equation}
Our goal is to prove that $u = y =0 \; \text{in} \; (0, L_{0})$ and $ \varphi = \psi= 0 \; \text{in} \; (L_{0}, L)$. For simplicity, we divide the proof into two steps.\\
\textbf{Step 1.} The aim of this step is to show that $\varphi =\psi = 0 \; \text{in}\; (L_{0},L)$. So, using (\ref{t}) and the third equation in (\ref{moo}), we have
\begin{equation}\label{dam}
\varphi=0 \;\; \;\; \text{in} \;\; \;\;  (\beta_{2},\beta_{4}).
\end{equation} 
From (\ref{mu})$_{3}$, (\ref{d2}) and the above equation, we get 
\begin{equation}\label{ml}
\psi = 0 \;\; \;\; \text{in} \;\; \;\; (\beta_{2},\beta_{3}).
\end{equation}
Using the above result and (\ref{dam}), we have $\varphi=\psi=0 \;  \text{in}  \; (\beta_{2},\beta_{3})$. Since $\varphi, \psi \in \mathnormal{H}^{2}(L_{0},L) \subset C^{1}([\beta_{2},\beta_{3}])$, then
\begin{equation}
\varphi(\zeta)=\varphi_{x}(\zeta)=\psi(\zeta)=\psi_{x}(\zeta)=0,\;\; \;\; \forall \zeta \in \left\lbrace \beta_{2},\beta_{3}\right\rbrace . 
\end{equation}
Let $V_{1}= (\varphi, \varphi_{x}, \psi, \psi_{x})^{\mathsf{T}}$, the following system
\begin{equation}\label{muo}
\left \{
\begin{aligned}
&\displaystyle \lambda^{2}\varphi+ a_{2}  \varphi_{xx}-\mathrm{i} \lambda c_{2}(x)\psi=0, &x \in &\; (L_{0},\beta_{2}) ,\\
&\displaystyle \lambda^{2}\psi+  \psi_{xx}+\mathrm{i} \lambda c_{2}(x)\varphi=0, &x \in &\; (L_{0},\beta_{2}) ,\\
&\displaystyle \varphi(\beta_{2})=\varphi_{x}(\beta_{2})=\psi(\beta_{2})=\psi_{x}(\beta_{2})=0,
\end{aligned} \right. 
\end{equation}
can be written as
\begin{equation}\label{kokoo}
\begin{cases}
\left( V_{1}\right)_{x} = \mathnormal{B}_{1}V_{1} \;\; \;\; \text{in} \;\; \;\;(L_{0},\beta_{2}),
\\
V_{1}\left( \beta_{2}\right)=0 , 
\end{cases} 
\end{equation}
where 
\begin{equation*}
\mathnormal{B}_{1}=\begin{pmatrix}
0&&&1&&&0&&&0  \\
\frac{-\lambda^{2}}{a_{2}}&&&0&&&\frac{\mathrm{i}\lambda}{a_{2}}c_{2}&&&0\\
0&&&0&&&0&&&1\\
-\mathrm{i}\lambda c_{2}&&&0&&&-\lambda^{2}&&&0
\end{pmatrix} .
\end{equation*}
The solution of the differential equation (\ref{kokoo}) is given by
\begin{equation}
V_{1}(x)=e^{\mathnormal{B}_{1}\left( x-\beta_{2}\right) }V_{1}\left( \beta_{2}\right) =0 \;\; \;\; \text{in} \;\; \;\; \left(L_{0},\beta_{2}\right) . 
\end{equation}
Then,
\begin{equation}\label{wa}
\varphi=\psi=0 \;\; \;\; \text{in} \;\; \;\;  (L_{0},\beta_{2}).
\end{equation}
Now, we still need to prove that $\varphi=0 \; \text{in} \; (\beta_{4},L)$ and $\psi=0 \; \text{in} \; (\beta_{3},L)$. Using (\ref{dam}), (\ref{ml}), the fact that $\varphi \in  C^{1}([\beta_{2},\beta_{4}])$, $\psi \in  C^{1}([\beta_{2},\beta_{3}])$ and the equations (\ref{mu})$_{3}$ and (\ref{mu})$_{4}$, we get the following systems:
\begin{equation}\label{muu}
\left \{
\begin{aligned}
&\displaystyle \lambda^{2}\varphi+ a_{2}  \varphi_{xx}=0, &x \in &\; (\beta_{4},L) ,\\
&\displaystyle \varphi(\beta_{4})=\varphi_{x}(\beta_{4})=\varphi(L)=0,
\end{aligned} \right. 
\end{equation}
and
\begin{equation}\label{muuo}
\left \{
\begin{aligned}
&\displaystyle \lambda^{2}\psi+ \psi_{xx}=0, &x \in &\; (\beta_{3},L) ,\\
&\displaystyle \psi(\beta_{3})=\psi_{x}(\beta_{3})=\psi(L)=0.
\end{aligned} \right. 
\end{equation}
Then, using Holmgren uniqueness theorem, we get
\begin{equation}
\varphi=0 \;\;  \text{in}  \;\; (\beta_{4},L) \;\; \;\; \text{and} \;\; \;\;  \psi=0\;\;  \text{in}  \;\; (\beta_{3},L).
\end{equation} 
Hence from (\ref{dam}), (\ref{ml}), (\ref{wa}) and the above result, we obtain
\begin{equation}\label{dou}
\varphi=\psi=0 \;\; \;\; \text{in} \;\; \;\; (L_{0},L).
\end{equation}
\textbf{Step2.} The aim of this step is to show that $u = y = 0 \; \text{in} \; (0,L_{0})$.
From (\ref{dou}) and the fact that $\varphi$, $\psi \in C^{1}\left( \left[ L_{0},L\right] \right) $, we have the following boundary condition
\begin{equation}
\varphi(L_{0})=\varphi_{x}(L_{0})=\psi(L_{0})=\psi_{x}(L_{0})=0.
\end{equation}
So, from (\ref{tran}) and the above equation, we obtain  
\begin{equation}\label{ff}
u(L_{0})=u_{x}(L_{0})=y(L_{0})=y_{x}(L_{0})=0. 
\end{equation}
From equations (\ref{mu})$_{1}$ and (\ref{mu})$_{2}$ and the fact that $c(x)=0 \;\;  \text{in} \;\; (\alpha_{3}, L_{0})$, we obtain the following system:
\begin{equation}\label{sy}
\left \{
\begin{aligned}
&\displaystyle \lambda^{2}u+ a_{1}  u_{xx} =0, &x \in &\; (\alpha_{3},L_{0}) ,\\
&\displaystyle \lambda^{2}y+  y_{xx} =0, &x \in &\; (\alpha_{3},L_{0}) .
\end{aligned} \right. 
\end{equation}
It is easy to see that system (\ref{sy}) admits only a trivial solution on $(\alpha_{3},L_{0})$ under the boundary condition (\ref{ff}). Then 
\begin{equation}\label{00}
u=y=0 \;\; \;\; \text{in} \;\; \;\; (\alpha_{3},L_{0}).
\end{equation}
Using the above equation and the fact that  $\varphi$, $\psi \in C^{1}\left( \left[ L_{0},L\right] \right) $, we have 
\begin{equation}\label{fff}
u(\alpha_{3})=u_{x}(\alpha_{3})=y(\alpha_{3})=y_{x}(\alpha_{3})=0. 
\end{equation}
Next, using equations (\ref{mu})$_{1}$ and (\ref{mu})$_{2}$, we get
\begin{equation}\label{muno}
\left \{
\begin{aligned}
&\displaystyle \lambda^{2}u+ a_{1}  u_{xx} - c_{1}(x)y=0, &x \in &\; (\alpha_{1},\alpha_{3}) ,\\
&\displaystyle \lambda^{2}y+  y_{xx}-c_{1}(x)u=0, &x \in &\; (\alpha_{1},\alpha_{3}) .
\end{aligned} \right. 
\end{equation}
Let $V_{2}= (u, u_{x}, y, y_{x})^{\mathsf{T}}$. From (\ref{fff}), $V_{2}(\alpha_{3})=0$. The system (\ref{muno})
can be written as the following equation
\begin{equation}\label{kookoo}
\left( V_{2}\right)_{x} = B_{2}V_{2} \;\; \;\; \text{in} \;\; \;\;(\alpha_{1},\alpha_{3}).
\end{equation}
where 
\begin{equation*}
B_{2}=\begin{pmatrix}
0&&&1&&&0&&&0  \\
\frac{-\lambda^{2}}{a_{1}}&&&0&&&\frac{c_{1}}{a_{1}}&&&0\\
0&&&0&&&0&&&1\\
 c_{1}&&&0&&&-\lambda^{2}&&&0
\end{pmatrix} .
\end{equation*}
The solution of the differential equation (\ref{kookoo}) is given by
\begin{equation}
V_{2}(x)=e^{B_{2}\left( x-\alpha_{3}\right) }V_{2}\left( \alpha_{3}\right) =0 \;\; \;\; \text{in} \;\; \;\; \left(\alpha_{1},\alpha_{3}\right) . 
\end{equation}
Then, 
\begin{equation}\label{lo}
u=y=0\;\;  \text{in}  \;\; (\alpha_{1},\alpha_{3}).
\end{equation}
Now, we still need to show that $u=y=0 \; \text{in} \; (0,\alpha_{1})$. Using the above result, the fact that $u, y \in  C^{1}([\alpha_{1},L_{0}])$ and the equations (\ref{mu})$_{1}$ and (\ref{mu})$_{2}$, we get:
\begin{equation}\label{muu}
\left \{
\begin{aligned}
&\displaystyle \lambda^{2}u+ a_{1}  u_{xx}=0, &x \in &\; (0,\alpha_{1}) ,\\
&\displaystyle u(\alpha_{1})=u_{x}(\alpha_{1})=u(0)=0,
\end{aligned} \right. 
\end{equation}
and
\begin{equation}\label{muuo}
\left \{
\begin{aligned}
&\displaystyle \lambda^{2}y+ y_{xx}=0, &x \in &\; (0,\alpha_{1}) ,\\
&\displaystyle y(\alpha_{1})=y_{x}(\alpha_{1})=y(0)=0.
\end{aligned} \right. 
\end{equation}
Again, using Holmgren uniqueness theorem, we have
\begin{equation}\label{aj}
u=y=0 \;\;  \text{in}  \;\; (0,\alpha_{1}) .
\end{equation} 
Finally, by using (\ref{moo}), (\ref{dou}), (\ref{00}), (\ref{lo}) and (\ref{aj})  we deduce that $U=0  \;\; \text{in}  \;\; (0,L)$ and we reached our disered result.
\end{proof}
\section{Exponential and polynomial stability}\label{exp}
In this section, we will study the exponential and polynomial stabilities of the system (\ref{pb})-(\ref{init}). Our main result in this part is the following theorems.
\begin{theorem}\label{port}
If $a_{2} = 1$, then the C$_{0}$-semigroup $\left( e^{t \mathcal{A}}\right) _{t\geq 0}$ is exponentially stable; i.e., there exists
constants $M\geq1$ and $\epsilon >0$ independent of $U_{0}$ such that
\begin{equation*}
	\left\| e^{t\mathcal{A}} U_{0} \right\|_{\mathcal{H}}\leq Me^{-\epsilon t} \left\| U_{0}\right\|_{\mathcal{H}} .
\end{equation*}
\end{theorem}
\begin{theorem}\label{portt}
If $a_{2} \neq 1$, then there exists
 $C>0$ such that for every $U_{0} \in \mathnormal{D}(\mathcal{A})$, we have 
	\begin{equation*}
	E(t)\leq\frac{C}{t} \left\| U_{0}\right\|^{2}_{\mathnormal{D}(\mathcal{A})}, \;\; \;\; \forall t>0.
	\end{equation*}
\end{theorem}
Since $\mathrm{i}\mathbb{R} \subset \rho(\mathcal{A})$ (see Sect.\ref{sec3}), according to Huang \cite{hu}, Pr$\ddot{u}$ss \cite{pr}, Borichev and Tomilov \cite{BT}, to proof Theorems\ref{port} and \ref{portt}, we still need to check if the following condition hold:
\begin{equation}\label{c2}
\sup_{\lambda \in \mathbb{R} } \frac{1}{\left| \lambda \right|^{\ell} }\left\| \left( \mathrm{i}\lambda I-\mathcal{A}\right) ^{-1}\right\| _{\mathcal{L}(\mathcal{H})}< \infty \;\;  \text{with} \;\;  \ell= \left \{
\begin{aligned}
0 & \;\; \text{for Theorem\ref{port}},
\\ 1 & \;\; \text{for Theorem\ref{portt}}. 
\end{aligned} \right.
\end{equation}
We will prove condition (\ref{c2}) by an argument of contradiction. For this purpose, suppose that (\ref{c2}) is false, then there exists  $\left\lbrace \left( \lambda_{n},U_{n}:=(u_{n},v_{n},y_{n},z_{n},\varphi_{n},\eta_{n},\psi_{n},\xi_{n})^{\mathsf{T}}\right) \right\rbrace  \subset \mathbb{R}^{*} \times \mathnormal{D}(\mathcal{A})$ with
\begin{equation}\label{hn}
\left| \lambda_{n} \right| \longrightarrow +\infty \;\; \text{and} \;\; \left\| U_{n}\right\|_{\mathcal{H}}=1 ,
\end{equation} 
such that
\begin{equation}\label{hnn}
\left( \lambda_{n}\right) ^{\ell}\left( i\lambda_{n} I -\mathcal{A}\right) U_{n}=F_{n}:=\left( f_{1,n},f_{2,n},f_{3,n},f_{4,n}, f_{5,n},f_{6,n},f_{7,n},f_{8,n}\right)^{\mathsf{T}}\longrightarrow 0 \;\; \text{in} \;\; \mathcal{H}.
\end{equation} 
For simplicity, we drop the index n. Equivalently, from (\ref{hnn}), we have
\begin{align} \label{lpo}
\mathrm{i}\lambda u-v&=\lambda^{-\ell}f_{1}\longrightarrow 0\;\; \text{in} \;\; \mathnormal{H}_{L}^{1}(0,L_{0}),\\ \label{ap}
\mathrm{i}\lambda v-a_{1}  u_{xx}+c_{1}(x) y &=\lambda^{-\ell}f_{2}\longrightarrow 0\;\; \text{in} \;\; \mathnormal{L}^{2}(0,L_{0}),\\ \label{fp}
\mathrm{i}\lambda y-z&=\lambda^{-\ell}f_{3}\longrightarrow 0\;\; \text{in} \;\; \mathnormal{H}_{L}^{1}(0,L_{0}) , \\ \label{aap}  
\mathrm{i}\lambda z - y_{xx}+c_{1}(x)u&=\lambda^{-\ell}f_{4}\longrightarrow 0\;\; \text{in} \;\; \mathnormal{L}^{2}(0,L_{0}), \\\label{aP}
\mathrm{i}\lambda \varphi-\eta&=\lambda^{-\ell}f_{5}\longrightarrow 0\;\; \text{in} \;\; \mathnormal{H}_{R}^{1}(L_{0},L),\\ \label{ap6}
\mathrm{i}\lambda \eta-a_{2}  \varphi_{xx}+c_{2}(x)\xi+d_{2}(x) \eta &=\lambda^{-\ell}f_{6}\longrightarrow 0\;\; \text{in} \;\; \mathnormal{L}^{2}(L_{0},L),\\ \label{fp7}
\mathrm{i}\lambda \psi-\xi&=\lambda^{-\ell}f_{7}\longrightarrow 0\;\; \text{in} \;\; \mathnormal{H}_{R}^{1}(L_{0},L) , \\ \label{aap8}  
\mathrm{i}\lambda \xi - \psi_{xx}-c_{2}(x)\eta&=\lambda^{-\ell}f_{8}\longrightarrow 0\;\; \text{in} \;\; \mathnormal{L}^{2}(L_{0},L). 
\end{align}
Here we will check the condition (\ref{c2}) by finding a contradiction with (\ref{hn}) by showing  $\left\| U\right\|_{\mathcal{H}}=o(1) $. From (\ref{hn}), (\ref{lpo}), (\ref{fp}), (\ref{aP}) and (\ref{fp7}), we obtain
\begin{equation}\label{ref1}
\left\|\lambda u \right\|_{\mathnormal{L}^{2}(0,L_{0})}=O(1), \;\; \;\;  \left\|\lambda y \right\|_{\mathnormal{L}^{2}(0,L_{0})}=O(1). 
\end{equation}
and
\begin{equation}\label{ref2}
\left\|\lambda \varphi \right\|_{\mathnormal{L}^{2}(L_{0},L)}=O(1), \;\; \;\; \left\|\lambda \psi \right\|_{\mathnormal{L}^{2}(L_{0},L)}=O(1) .
\end{equation}
For clarity, we will divide the proof into several lemmas.
\begin{lemma}\label{lemme1}
The solution $(u, v, y, z,\varphi, \eta, \psi, \xi) \in  \mathnormal{D}(\mathcal{A})$ of (\ref{lpo})-(\ref{aap8}) satisfies the following asymptotic behavior estimation 
 \begin{equation}\label{lemm}
\int_{\beta_{2}}^{\beta_{4}} \left| \eta \right| ^{2} dx= o\left( \lambda ^{- \ell}\right) \;\; \text{and} \;\; \int_{\beta_{2}}^{\beta_{4}} \left|\lambda \varphi \right| ^{2} dx= o\left(\lambda ^{- \ell}\right).
	\end{equation}
\end{lemma}
\begin{proof}
	Taking the inner product of (\ref{hnn}) with $U$ in $\mathcal{H}$ and using (\ref{re}), we get
	$$
	\int_{\beta_{2}}^{\beta_{4} } \left| \eta \right|^{2}dx=-\Re \left( \langle \mathcal{A}U,U\rangle_{\mathcal{H}}\right) = \Re \left( \langle\left( \mathrm{i}\lambda I- \mathcal{A}\right) U,U\rangle_{\mathcal{H}}\right)=\lambda^{-\ell}\Re  \langle F,U\rangle_{\mathcal{H}}\leqslant \left| \lambda\right| ^{-\ell} \left\|F \right\|_{\mathcal{H}} \left\|U \right\|_{\mathcal{H}} .
	$$
Thus, from the above equation, the fact that $\left\| F\right\|_{\mathcal{H}}=o(1) $ and $\left\| U\right\|_{\mathcal{H}}=1 $, we obtain the first estimation in (\ref{lemm}). By using (\ref{aP}) and the first estimation in (\ref{lemm}), we get the last estimation.	
\end{proof}
Inserting (\ref{aP}) and (\ref{fp7}) into (\ref{ap6}) and (\ref{aap8}), we get the following system
\begin{align}\label{knn}
\lambda^{2}\varphi+a_{2} \varphi_{xx} -\mathrm{i}\lambda d_{2}(x)\varphi-\mathrm{i}\lambda c_{2}(x)\psi&=F_{1},\\\label{knnn}
\lambda^{2}\psi+\psi_{xx}+\mathrm{i}\lambda c_{2}(x)\varphi&=F_{2},
\end{align}
where
\begin{equation*}
\left \{
\begin{aligned}
&F_{1}=-\mathrm{i}\lambda^{-\ell+1} f_{5}-d_{2}(x)\lambda^{-\ell} f_{5}-\lambda^{-\ell} f_{6}-\lambda^{-\ell} c_{2}(x)f_{7}, \\
&F_{2}=\lambda^{-\ell}c_{2}(x)f_{5}-\mathrm{i}\lambda^{-\ell+1} f_{7}-\lambda^{-\ell} f_{8}.\\
\end{aligned} \right.
\end{equation*}
\begin{lemma}\label{lemme2}
Let $\max\left( \beta_{2},\frac{\beta_{3}-\beta_{2}}{5}\right)<\delta< \beta_{4} $. The solution $(u, v, y, z,\varphi, \eta, \psi, \xi) \in  \mathnormal{D}(\mathcal{A})$ of (\ref{lpo})-(\ref{aap8}) satisfies the following asymptotic behavior estimation
	\begin{equation}\label{lemm2}
	\int_{\beta_{2}+\delta}^{\beta_{4}-\delta} \left| \varphi_{x} \right| ^{2} dx= o\left( \lambda ^{-\min \left(\ell,\frac{\ell}{2}+1 \right) } \right)  .
	\end{equation}
\end{lemma}
\begin{proof}
First, we define the cut-off function $\theta_{1} \in C^{1}(L_{0},L)$ by 
\begin{equation}\label{a9}
0\leq \theta_{1} \leq 1, \;\; \theta_{1}=1 \;
\text{on} \;(\beta_{2}+\delta, \beta_{4}-\delta) \;\; \;\; \text{and} \;\; \;\; \theta_{1}=0 \; \text{on} \; (L_{0},\beta_{2}) \cup (\beta_{4},L). 
\end{equation}
Multiplying (\ref{knn}) by $\theta_{1} \overline{\varphi}$, integrating over $(L_{0},L)$ and taking the real part, we get
\begin{align}\label{Rr}
&\int_{L_{0}}^{L} \theta_{1} \left| \lambda \varphi \right|^{2} dx - a_{2} \int_{L_{0}}^{L} \theta_{1} \left| \varphi_{x} \right|^{2} dx - \Re\left\lbrace  a_{2}\int_{L_{0}}^{L}\theta_{1}' \varphi_{x} \overline{\varphi} dx\right\rbrace  \\ \nonumber
& - \Re\left\lbrace  \mathrm{i} \lambda c_{2} \int_{\beta_{2}}^{\beta_{3}}  \theta_{1} \psi  \overline{\varphi} dx\right\rbrace =\Re\left\lbrace   \int_{L_{0}}^{L}  F_{1} \theta_{1}  \overline{\varphi} dx\right\rbrace. 
\end{align}
Using the fact that $\left\|F \right\|\rightarrow 0 $ in $\mathcal{H}$ and $\lambda \varphi$ is uniformly bounded in $\mathnormal{L}^{2}(L_{0},L)$, we obtain 
\begin{align} \label{at}
\Re\left\lbrace   \int_{L_{0}}^{L}  F_{1} \theta_{1}  \overline{\varphi} dx\right\rbrace&= \Re\left\lbrace   \int_{L_{0}}^{L}  \theta_{1} \left( -\mathrm{i}\lambda^{-\ell+1} f_{5}-d_{2}(x)\lambda^{-\ell} f_{5}-\lambda^{-\ell} f_{6}-\lambda^{-\ell} c_{2}(x)f_{7} \right)  \overline{\varphi}  dx\right\rbrace \\ \nonumber
& = o\left( \lambda^{-\ell}\right) .
\end{align} 
On the other hand, using Lemma\ref{lemme1}, the fact that $\lambda \psi$ and $\varphi_{x}$ are uniformly bounded in $\mathnormal{L}^{2}(L_{0},L)$ and the definition of $\theta_{1}$, we get
\begin{equation}\label{R}
\Re\left\lbrace  \mathrm{i} \lambda c_{2} \int_{\beta_{2}}^{\beta_{3}}  \theta_{1} \psi \overline{\varphi} dx\right\rbrace  =o\left( \lambda ^{-\left( \frac{\ell}{2}+1\right) }\right),  
\end{equation}
and
\begin{equation}\label{R2}
 \Re\left\lbrace  a_{2}\int_{L_{0}}^{L}\theta_{1}' \varphi_{x} \overline{\varphi}  dx\right\rbrace   =o\left(  \lambda ^{-\left( \frac{\ell}{2}+1\right) }\right). 
\end{equation}
Furthermore, using Lemma\ref{lemme1} and the definition of the function $\theta_{1}$
in (\ref{a9}), we get
\begin{equation}\label{R3}
\int_{L_{0}}^{L}\theta_{1} \left| \lambda \varphi \right|^{2} dx  =o\left( \lambda ^{-\ell}\right). 
\end{equation}
Inserting (\ref{at})-(\ref{R3}) in (\ref{Rr}), we get (\ref{lemm2}).
\end{proof}
\begin{lemma}\label{lemme3}
	Let  $\max\left( \beta_{2},\frac{\beta_{3}-\beta_{2}}{5}\right)<\delta< \beta_{4} $. The solution $(u, v, y, z,\varphi, \eta, \psi, \xi) \in  \mathnormal{D}(\mathcal{A})$ of (\ref{lpo})-(\ref{aap8}) satisfies the following asymptotic behavior estimation
	\begin{equation}\label{lemm3}
	\int_{\beta_{2}+2\delta}^{\beta_{3}-\delta}   \left| \psi_{x} \right| ^{2} dx 
	\leq  \frac{ \left|  a_{2}-1 \right|  \left|\lambda \right| }{\left| c_{2}\right| } o\left(  \lambda  ^{-\min\left(\frac{\ell}{2},\frac{\ell}{4}+\frac{1}{2} \right) }\right)+o(1). 
	\end{equation}
\end{lemma}
\begin{proof}
First, we fix a cut-off function $\theta_{2} \in C^{1}(L_{0},L)$ such that $0\leq \theta_{2}(x) \leq1$, for all $x \in \left[L_{0},L\right]  $ and 
\begin{equation*}
\theta_{2}(x)=
\begin{cases}
1 & \text{ if} \;  x \in \left[ \beta_{2}+2\delta ,\beta_{3}-\delta\right] , 
\\
0 &  \text{ if }\;    x \in \left[ L_{0} ,\beta_{2}+\delta\right] \cup   \left[ \beta_{3} ,L\right] .
\end{cases} 
\end{equation*}
From (\ref{aap8}), $\mathrm{i}\lambda^{-1}\theta_{2}\overline{\psi_{xx}}$ is uniformly bounded in $\mathnormal{L}^{2}(L_{0}, L)$. Multiplying (\ref{ap6}) by $\mathrm{i}\lambda^{-1}\theta_{2}\overline{\psi_{xx}}$, using integration by parts over $(L_{0}, L)$ and using the fact that $\left\| f_{6}\right\| _{\mathnormal{L}^{2}(L_{0}, L)} = o(1)$, we get
\begin{align} \label{4.01} 
&\int_{L_{0}}^{L}\theta_{2}' \eta \overline{\psi_{x}} dx+ \int_{L_{0}}^{L}\theta_{2} \eta_{x} \overline{\psi_{x}} dx- \frac{\mathrm{i}}{\lambda}
a_{2}\int_{L_{0}}^{L}\theta_{2} \varphi_{xx} \overline{\psi_{xx}} dx\\ \nonumber
& - \frac{\mathrm{i}}{\lambda}c_{2}
\int_{\beta_{2}+\delta}^{\beta_{3}} \left( \theta_{2}'\xi+\theta_{2}\xi_{x}\right)  \overline{\psi_{x}} dx  
- \frac{\mathrm{i}}{\lambda}d_{2}
\int_{\beta_{2}+\delta}^{\beta_{3}} \left( \theta_{2}'\eta+ \theta_{2} \eta_{x}\right)   \overline{\psi_{x}} dx=o\left(  \lambda^{-\ell}\right) .
\end{align}
From (\ref{aP}) and (\ref{fp7}), we obtain
$$
\eta_{x}= \mathrm{i}\lambda \varphi_{x}-\lambda^{-\ell}\left( f_{5}\right) _{x} \;\; \;\; \text{and} \;\; \;\; - \frac{\mathrm{i}}{\lambda}\xi_{x}=  \psi_{x}+\mathrm{i}\lambda^{-(\ell+1)}\left( f_{7}\right)_{x}.
$$
Inserting the above equations in (\ref{4.01}) and taking the real part, we get
\begin{align}\label{H} \nonumber
&\Re\left\lbrace \mathrm{i}\lambda \int_{L_{0}}^{L}\theta_{2} \varphi_{x} \overline{\psi_{x}} dx\right\rbrace - \Re\left\lbrace \frac{\mathrm{i}}{\lambda} a_{2} \int_{L_{0}}^{L}\theta_{2} \varphi_{xx} \overline{\psi_{xx}} dx\right\rbrace 
+c_{2} \int_{\beta_{2}+\delta}^{\beta_{3}} \theta_{2} \left|\psi_{x}\right|^{2} dx\\
&= - \Re\left\lbrace  \int_{L_{0}}^{L} \theta_{2}'\eta \overline{\psi_{x}} dx\right\rbrace + \Re\left\lbrace \frac{1}{\lambda^{\ell}} \int_{L_{0}}^{L} \theta_{2} \left( f_{5}\right) _{x}  \overline{\psi_{x}} dx\right\rbrace   + \Re\left\lbrace \frac{\mathrm{i}}{\lambda} c_{2}\int_{\beta_{2}+\delta}^{\beta_{3}}  \theta_{2}' \xi \overline{\psi_{x}} dx\right\rbrace \\ \nonumber
&- \Re\left\lbrace \frac{\mathrm{i}}{\lambda^{(\ell+1)}} c_{2} \int_{\beta_{2}+\delta}^{\beta_{3}}  \theta_{2} \left( f_{7}\right)_{x} \overline{\psi_{x}} dx\right\rbrace + \Re\left\lbrace \frac{\mathrm{i}}{\lambda} d_{2} \int_{\beta_{2}+\delta}^{\beta_{3}}  \theta_{2}' \eta \overline{\psi_{x}} dx \right\rbrace \\ \nonumber
&- \Re\left\lbrace  d_{2} \int_{\beta_{2}+\delta}^{\beta_{3}}  \theta_{2} \varphi_{x} \overline{\psi_{x}} dx \right\rbrace - \Re\left\lbrace \frac{\mathrm{i}}{\lambda^{(\ell+1)}} d_{2} \int_{\beta_{2}+\delta}^{\beta_{3}}  \theta_{2} \left( f_{5}\right) _{x} \overline{\psi_{x}} dx \right\rbrace + o\left( \lambda ^{-\ell}\right) .
\end{align}
Using the fact that $ \psi_{x}$ is uniformly bounded in $\mathnormal{L}^{2}(L_{0}, L)$ and $\left\| F\right\| _{\mathcal{H}} = o(1)$, we get
\begin{align}\label{Ha}
&\Re\left\lbrace \frac{1}{\lambda^{\ell}} \int_{L_{0}}^{L}  \theta_{2} \left( f_{5}\right) _{x} \overline{\psi_{x}} dx\right\rbrace = o\left( \lambda^{-\ell}\right) , \; -\Re\left\lbrace \frac{\mathrm{i}}{\lambda^{(\ell+1)}} c_{2} \int_{\beta_{2}+\delta}^{\beta_{3}}  \theta_{2} \left( f_{7}\right)_{x} \overline{\psi_{x}} dx\right\rbrace = o\left(  \lambda^{-(\ell+1)}\right)\\ \nonumber
& \;\ \;\; \;\; \;\; \;\; \;\; \text{and}\;\;  \;-\Re\left\lbrace \frac{\mathrm{i}}{\lambda^{(\ell+1)}} d_{2} \int_{\beta_{2}+\delta}^{\beta_{3}}  \theta_{2} \left( f_{5}\right) _{x} \overline{\psi_{x}} dx\right\rbrace = o\left(  \lambda^{-(\ell+1)}\right)  .
\end{align}
Next, using the fact that $ \psi_{x}$, $ \eta$ and $\xi$ are uniformly bounded in $\mathnormal{L}^{2}(L_{0}, L)$, we have
\begin{equation}
\Re\left\lbrace \frac{\mathrm{i}}{\lambda} c_{2} \int_{\beta_{2}+\delta}^{\beta_{3}}  \theta_{2}' \xi \overline{\psi_{x}} dx\right\rbrace =O\left( \lambda ^{-1}\right) = o(1) \;\; \text{and} \;\; 
\Re\left\lbrace \frac{\mathrm{i}}{\lambda} d_{2} \int_{\beta_{2}+\delta}^{\beta_{3}}  \theta_{2}' \eta \overline{\psi_{x}} dx\right\rbrace =O\left( \lambda ^{-1}\right) = o(1)  .
\end{equation}
Furthermore, using the estimations (\ref{lemm}), (\ref{lemm2}) and the fact that $ \psi_{x}$ is uniformly bounded in $\mathnormal{L}^{2}(L_{0}, L)$, we get 
\begin{equation}\label{Has}
-\Re\left\lbrace  \int_{L_{0}}^{L}  \theta_{2}' \eta \overline{\psi_{x}} dx\right\rbrace = o\left(  \lambda^{-\frac{\ell}{2}}\right) \
  \;\;  \text{and} \;\;   -\Re\left\lbrace  d_{2} \int_{\beta_{2}+\delta}^{\beta_{3}}  \theta_{2} \varphi_{x} \overline{\psi_{x}} dx\right\rbrace =o\left(  \lambda  ^{-\min\left(\frac{\ell}{2},\frac{\ell}{4}+\frac{1}{2} \right) }\right).
\end{equation}
Inserting (\ref{Ha})-(\ref{Has}) in (\ref{H}) and using the definition of $(\theta_{2})$, we get
\begin{equation}\label{rs1}
c_{2}\int_{\beta_{2}+2\delta}^{\beta_{3}-\delta} \left| \psi_{x}\right| ^{2} dx +\Re\left\lbrace \mathrm{i}\lambda \int_{\beta_{2}+\delta}^{\beta_{3}} \theta_{2} \varphi_{x} \overline{\psi_{x}} dx\right\rbrace - \Re\left\lbrace \frac{\mathrm{i}}{\lambda} a_{2} \int_{\beta_{2}+\delta}^{\beta_{3}} \theta_{2} \varphi_{xx}\overline{\psi_{xx}} dx\right\rbrace=o(1).
\end{equation}
 From (\ref{ap6}), $\mathrm{i}\lambda^{-1}a_{2}\theta_{2}\overline{\varphi_{xx}}$ is uniformly bounded in $\mathnormal{L}^{2}(L_{0}, L)$. Multiplying (\ref{aap8}) by $\mathrm{i}\lambda^{-1}a_{2}\theta_{2}\overline{\varphi_{xx}}$, using integration by parts over $(L_{0}, L)$ and the fact that $\left\| f_{8}\right\| _{\mathnormal{L}^{2}(L_{0}, L)} = o(1)$, we get
\begin{align} \label{4.02} 
&a_{2}\int_{L_{0}}^{L}\theta_{2}'\xi \overline{\varphi_{x}} dx+ a_{2} \int_{L_{0}}^{L}\theta_{2} \xi_{x} \overline{\varphi_{x}} dx- \frac{\mathrm{i}}{\lambda}
a_{2}\int_{L_{0}}^{L}\theta_{2} \psi_{xx} \overline{\varphi_{xx}} dx \\ \nonumber
& + \frac{\mathrm{i}}{\lambda} a_{2} c_{2}
\int_{\beta_{2}+\delta}^{\beta_{3}}   \theta_{2}' \eta  \overline{\varphi_{x}} dx 
+ \frac{\mathrm{i}}{\lambda} a_{2} c_{2} \int_{\beta_{2}+\delta}^{\beta_{3}}   \theta_{2} \eta_{x}  \overline{\varphi_{x}} dx = o\left( \lambda^{-\ell}\right) .
\end{align}
From (\ref{aP}) and (\ref{fp7}), we obtain
$$
\eta_{x}= \mathrm{i}\lambda \varphi_{x}-\lambda^{-\ell}\left( f_{5}\right) _{x} \;\; \;\; \text{and} \;\; \;\; \xi_{x}= \mathrm{i} \lambda  \psi_{x}-\lambda^{-\ell}\left( f_{7}\right)_{x}.
$$
Inserting the above equations in (\ref{4.02}) and taking the real part, we get
\begin{align}\label{A} \nonumber
&\Re\left\lbrace \mathrm{i}\lambda a_{2} \int_{L_{0}}^{L}\theta_{2} \psi_{x} \overline{\varphi_{x}} dx\right\rbrace - \Re\left\lbrace \frac{\mathrm{i}}{\lambda} a_{2} \int_{L_{0}}^{L}\theta_{2} \psi_{xx} \overline{\varphi_{xx}} dx\right\rbrace =-\Re\left\lbrace   a_{2} \int_{L_{0}}^{L} \theta_{2}' \xi \overline{\varphi_{x}} dx\right\rbrace\\
& + \Re\left\lbrace \frac{a_{2}}{\lambda^{\ell}}  \int_{L_{0}}^{L} \theta_{2} \left( f_{7}\right) _{x} \overline{\varphi_{x}} dx\right\rbrace - \Re\left\lbrace \frac{\mathrm{i}}{\lambda} a_{2} c_{2} \int_{\beta_{2}+\delta}^{\beta_{3}} \theta_{2}' \eta  \overline{\varphi_{x}} dx\right\rbrace \\ \nonumber
&+ \Re\left\lbrace \frac{\mathrm{i}}{\lambda^{(\ell+1)}}a_{2} c_{2} \int_{\beta_{2}+\delta}^{\beta_{3}}  \theta_{2} \left( f_{5}\right)_{x}\overline{\varphi_{x}} dx\right\rbrace +
a_{2}  c_{2} \int_{\beta_{2}+\delta}^{\beta_{3}} \theta_{2} \left| \varphi_{x} \right|^{2} dx + o\left( \lambda^{-\ell}\right). 
\end{align}
First, using the fact that $ \varphi_{x}$ is uniformly bounded in $\mathnormal{L}^{2}(L_{0}, L)$ and $\left\| F \right\| _{\mathcal{H}} = o(1)$, we get
\begin{align}\label{Aa}
&\Re\left\lbrace \frac{a_{2}}{\lambda^{\ell}} \int_{L_{0}}^{L}  \theta_{2} \left( f_{7}\right) _{x} \overline{\varphi_{x}} dx\right\rbrace = o\left( \lambda^{-\ell}\right)  \;\; \;\;  \text{and} \;\; \\ \nonumber
&\Re\left\lbrace \frac{\mathrm{i}}{\lambda^{(\ell+1)}} a_{2}c_{2} \int_{\beta_{2}+\delta}^{\beta_{3}} \theta_{2} \left( f_{5}\right)_{x} \overline{\varphi_{x}} dx\right\rbrace = o\left( \lambda ^{-(\ell+1)}\right).
\end{align}
Next, by the fact that $ \varphi_{x}$ and $\eta$ are uniformly bounded in $\mathnormal{L}^{2}(L_{0}, L)$, we have
\begin{equation}
-\Re\left\lbrace \frac{\mathrm{i}}{\lambda} a_{2} c_{2} \int_{\beta_{2}+\delta}^{\beta_{3}} \theta_{2}' \eta \overline{\varphi_{x}} dx\right\rbrace =O\left( \lambda^{-1}\right) =o(1)
\end{equation}
On the other hand, using the estimation (\ref{lemm2}), the fact that $ \xi$ is uniformly bounded in $\mathnormal{L}^{2}(L_{0}, L)$, we get  
\begin{align}\label{Aaa}
-\Re\left\lbrace a_{2} \int_{L_{0}}^{L}  \theta_{2}' \xi \overline{\varphi_{x}} dx\right\rbrace = o\left(  \lambda  ^{-\min \left( \frac{\ell}{2}, \frac{\ell}{4}+\frac{1}{2}\right) }\right)  \;\;   \text{and}  \;\;  a_{2} c_{2} \int_{\beta_{2}+\delta}^{\beta_{3}}  \theta_{2} \left|\varphi_{x}\right|^{2} dx = o\left(  \lambda^{-\min \left( \ell, \frac{\ell}{2}+1\right) }\right). 
\end{align}
Inserting (\ref{Aa})-(\ref{Aaa}) in (\ref{A}) and using the definition of $(\theta_{2})$, we get
\begin{equation}\label{rs2}
\Re\left\lbrace \mathrm{i} \lambda a_{2} \int_{\beta_{2}+\delta}^{\beta_{3}} \theta_{2} \psi_{x} \overline{\varphi_{x}}dx\right\rbrace   -\Re\left\lbrace \frac{\mathrm{i}}{\lambda} a_{2}  \int_{\beta_{2}+\delta}^{\beta_{3}} \theta_{2} \psi_{xx} \overline{\varphi_{xx}} dx\right\rbrace=o\left( 1\right) .
\end{equation}
Now, adding (\ref{rs1}) and (\ref{rs2}), we obtain
\begin{equation}
\int_{\beta_{2}+2\delta}^{\beta_{3}-\delta}   \left| \psi_{x} \right| ^{2} dx 
\leq  \frac{ \left|  a_{2}-1 \right|  \left|\lambda \right| }{\left| c_{2}\right| } \int_{\beta_{2}+\delta}^{\beta_{3}} \left| \varphi_{x}\right|  \left| \psi_{x}\right| dx+o(1). 
\end{equation}
Finally, using Lemma\ref{lemme2} and the fact that $ \psi_{x}$ is uniformly bounded in $\mathnormal{L}^{2}(L_{0}, L)$, we get our estimation (\ref{lemm3}).
\end{proof}
\begin{lemma}\label{lemme4}
	Let  $\max\left( \beta_{2},\frac{\beta_{3}-\beta_{2}}{5}\right)<\delta< \beta_{4} $. The solution $(u, v, y, z,\varphi, \eta, \psi, \xi) \in  \mathnormal{D}(\mathcal{A})$ of (\ref{lpo})-(\ref{aap8}) satisfies the following asymptotic behavior estimation
	\begin{equation}\label{lemm4}
	\int_{\beta_{2}+3\delta}^{\beta_{3}-2\delta}   \left| \xi\right| ^{2} dx 
	\leq   \frac{ 3\left|  a_{2}-1 \right|  \left|\lambda \right| }{\left| c_{2}\right| } o\left(  \lambda  ^{-\min\left(\frac{\ell}{2},\frac{\ell}{4}+\frac{1}{2} \right) }\right)+o(1). 
	\end{equation}
\end{lemma}
\begin{proof}
First, we fix a cut-off function $\theta_{3} \in C^{1}(L_{0},L)$ such that $0\leq \theta_{3}(x) \leq1$, for all $x \in \left[L_{0},L\right]  $ and 
\begin{equation*}
\theta_{3}(x)=
\begin{cases}
1 & \text{ if} \;  x \in \left[ \beta_{2}+3\delta ,\beta_{3}-2\delta\right],  
\\
0 &  \text{ if }\;    x \in \left[ L_{0} ,\beta_{2}+2\delta\right] \cup   \left[ \beta_{3} -\delta,L\right] .
\end{cases} 
\end{equation*}
Multiplying (\ref{aap8}) by $-\mathrm{i}\lambda^{-1}\theta_{3}\overline{\xi}$, using integration by parts over $(L_{0}, L)$, the fact that $\xi$ is uniformly bounded in $\mathnormal{L}^{2}(0,L)$ and $\left\| f_{8}\right\| _{\mathnormal{L}^{2}(L_{0}, L)} = o(1)$, we get
\begin{equation} \label{la} 
\int_{L_{0}}^{L}\theta_{3}\left| \xi \right|^{2}dx- \frac{\mathrm{i}}{\lambda} \int_{L_{0}}^{L}\theta_{3}' \psi_{x} \overline{\xi} dx- \frac{\mathrm{i}}{\lambda}
\int_{L_{0}}^{L}\theta_{3} \psi_{x} \overline{\xi_{x}} dx
 + \frac{\mathrm{i}}{\lambda}c_{2}
\int_{\beta_{2}+2\delta}^{\beta_{3}-\delta}  \theta_{3} \eta  \overline{\xi} dx=o\left(  \lambda ^{-(\ell+1)}\right) .
\end{equation}
From  (\ref{fp7}), we obtain
$$
 -\frac{\mathrm{i}}{\lambda}\overline{\xi_{x}}=  -\overline{\psi_{x}}+\mathrm{i}\lambda^{-(\ell+1)}\overline{\left( f_{7}\right)_{x}}.
$$
Inserting the above equation in (\ref{la}), we get
\begin{align}\label{sty} 
& \int_{L_{0}}^{L}\theta_{3} \left| \xi \right|^{2} dx= \int_{L_{0}}^{L}\theta_{3} \left| \psi_{x}\right| ^{2}dx - \mathrm{i} \lambda^{-(\ell+1)} \int_{L_{0}}^{L} \theta_{3} \overline{\left( f_{7}\right)_{x}} \psi_{x} dx \\ \nonumber
&+\frac{\mathrm{i}}{\lambda}  \int_{L_{0}}^{L} \theta_{3}' \psi_{x} \overline{\xi} dx - \frac{\mathrm{i}}{\lambda} c_{2} \int_{\beta_{2}+2\delta}^{\beta_{3}-\delta} \theta_{3} \eta \overline{\xi} dx+ o\left(  \lambda  ^{-(\ell+1)}\right).
\end{align}
Using the fact that $ \psi_{x}$,  $ \eta$ and  $ \xi$ are uniformly bounded in $\mathnormal{L}^{2}(L_{0}, L)$ and $\left\| \left( f_{7}\right)_{x}\right\| _{\mathnormal{L}^{2}(L_{0}, L)} = o(1)$, we get
\begin{align}\label{st}
&\left|  \mathrm{i}\lambda^{-(\ell+1)} \int_{L_{0}}^{L}  \theta_{3} \overline{\left( f_{7}\right) _{x}} \psi_{x} dx\right| =  o\left(  \lambda ^{-(\ell+1)}\right)  , \;  \left| \frac{\mathrm{i}}{\lambda}  \int_{L_{0}}^{L}  \theta_{3}' \psi_{x} \overline{\xi} dx\right|  =o\left( 1\right)\\ \nonumber
& \;\ \;\; \;\; \;\; \;\; \;\; \text{and}\;\;  \;
\left|  \frac{\mathrm{i}}{\lambda} c_{2} \int_{\beta_{2}+2\delta}^{\beta_{3}-\delta}  \theta_{3} \eta \overline{\xi} dx\right|  = o\left( 1\right)  .
\end{align}
Using (\ref{lemm3}) and the definition of $\theta_{3}$, we get 
\begin{align}\label{stt}
 \int_{L_{0}}^{L}  \theta_{3}\left|  \psi_{x} \right|^{2} dx &\leq 3 \int_{\beta_{2}+2\delta}^{\beta_{3}-\delta}  \left| \psi_{x} \right| ^{2} dx\\ \nonumber
& \leq \frac{3 \left|  a_{2}-1 \right|  \left|\lambda \right| }{\left| c_{2}\right| } o\left(  \lambda  ^{-\min\left(\frac{\ell}{2},\frac{\ell}{4}+\frac{1}{2} \right) }\right)+o(1).   
\end{align}
Inserting (\ref{st}) and (\ref{stt}) in (\ref{sty}) and using the definition of $\theta_{3}$, we get the desired estimation (\ref{lemm4}).
\end{proof}
Now, we fix a function $g\in C^{1}\left( \left[ \beta_{1}, \beta_{2}+3\delta\right] \right) $ such that
\begin{equation}
g\left( \beta_{1}\right) =-g\left( \beta_{2}+3\delta\right) =1, \;\;  \max_{x \in \left[\beta_{1},\beta_{2}+3\delta \right] }  \left| g(x)\right| =M_{g} \;\; \text{and} \;\;  \max_{x \in \left[\beta_{1},\beta_{2}+3\delta \right] }  \left| g'(x)\right| =M_{g'},
\end{equation}
where $M_{g}$ and $M_{g'}$ are strictly positive constant numbers.
\begin{remark}
It is easy to see the existence of $g(x)$. For example, we can take 
\begin{equation}
g(x)= \frac{1 }{\left( \beta_{2}+3\delta -\beta_{1}\right) ^{2}}\left(-2x^{2}+4\beta_{1}x+\left( \beta_{2}+3\delta\right) ^{2} -\beta_{1}^{2}-2\left( \beta_{2}+3\delta\right)\beta_{1}\right)  , 
\end{equation}
to get
$$
g\left( \beta_{1}\right) =-g\left( \beta_{2}+3\delta\right) =1, \;\; g\in C^{1}\left( \left[ \beta_{1}, \beta_{2}+3\delta\right] \right) ,
$$
and
$$  \max_{x \in \left[\beta_{1},\beta_{2}+3\delta \right] }  \left| g(x)\right| =1 \;\; \text{and} \;\;  \max_{x \in \left[\beta_{1},\beta_{2}+3\delta \right] }  \left| g'(x)\right| =\frac{4}{\beta_{2}+3\delta-\beta_{1}}.
$$
\end{remark}
\begin{lemma}\label{lemme5}
Let  $\max\left( \beta_{2},\frac{\beta_{3}-\beta_{2}}{5}\right)<\delta< \beta_{4} $. The solution $(u, v, y, z,\varphi, \eta, \psi, \xi) \in  \mathnormal{D}(\mathcal{A})$ of (\ref{lpo})-(\ref{aap8}) satisfies the following asymptotic behavior estimations
\begin{equation}\label{the1}
\left| \eta\left( \beta_{2}+3\delta\right) \right| ^{2} + \left| \eta \left( \beta_{1}\right) \right| ^{2} =O\left( \lambda \right), 
\end{equation}
and
\begin{equation}\label{the2}
\left| \xi \left( \beta_{2}+3\delta\right) \right| ^{2} + \left| \xi \left( \beta_{1}\right) \right| ^{2} + \left| \psi_{x}\left( \beta_{2}+3\delta\right) \right| ^{2} + \left| \psi_{x} \left( \beta_{1}\right) \right| ^{2} =O\left( 1 \right). 
\end{equation}
\end{lemma}
\begin{proof}
First, deriving (\ref{aP}) with respect to $x$, we have
\begin{equation}\label{Hl}
\mathrm{i}\lambda \varphi_{x}-\eta_{x}= \lambda^{-\ell}\left( f_{5}\right) _{x} .
\end{equation}	
Multiplying the above equation by $2g \overline{\eta}$, integrating over $ \left( \beta_{1},\beta_{2}+3\delta\right) $, then taking the real part, we get
\begin{align}\label{E1}
&\Re\left\lbrace 2 \mathrm{i} \lambda  \int_{\beta_{1}}^{\beta_{2}+3\delta}  g \varphi_{x} \overline{\eta} dx\right\rbrace - \int_{\beta_{1}}^{\beta_{2}+3\delta}  g \left(  \left| \eta \right| ^{2}\right) _{x} dx\\ \nonumber
& = 
\Re\left\lbrace 2  \lambda ^{-\ell} \int_{\beta_{1}}^{\beta_{2}+3\delta}  g \left( f_{5}\right) _{x} \overline{\eta} dx\right\rbrace.
\end{align}
Using integration by parts in (\ref{E1}), we get
\begin{align}\label{E11}
& \left[ -g\left| \eta \right| ^{2}\right] _{\beta_{1}}^{\beta_{2}+3\delta}= - \int_{\beta_{1}}^{\beta_{2}+3\delta}  g'  \left| \eta \right| ^{2} dx - \Re\left\lbrace 2 \mathrm{i} \lambda  \int_{\beta_{1}}^{\beta_{2}+3\delta}  g \varphi_{x} \overline{\eta} dx\right\rbrace\\ \nonumber
& + 
\Re\left\lbrace 2  \lambda ^{-\ell} \int_{\beta_{1}}^{\beta_{2}+3\delta}  g \left( f_{5}\right) _{x} \overline{\eta} dx\right\rbrace.
\end{align}
Using the definition of $g$ and Cauchy-Schwarz inequality in (\ref{E11}), we obtain
\begin{align}\label{EE} \nonumber
\left| \eta \left(\beta_{2}+3\delta \right) \right|^{2} +\left| \eta \left(\beta_{1} \right) \right|^{2} &\leq M_{g'} \int_{\beta_{1}}^{\beta_{2}+3\delta} \left| \eta \right|^{2} dx\\ 
&+2 \left| \lambda\right|  M_{g} \left(\int_{\beta_{1}}^{\beta_{2}+3\delta} \left| \varphi_{x} \right|^{2} dx \right) ^{\frac{1}{2}}
\left(\int_{\beta_{1}}^{\beta_{2}+3\delta} \left| \eta \right|^{2} dx \right) ^{\frac{1}{2}}\\ \nonumber 
&+2 \left| \lambda\right|^{-\ell}  M_{g} \left(\int_{\beta_{1}}^{\beta_{2}+3\delta} \left| \left( f_{5}\right)_{x}\right|^{2} dx \right) ^{\frac{1}{2}}
\left(\int_{\beta_{1}}^{\beta_{2}+3\delta} \left| \eta \right|^{2} dx \right) ^{\frac{1}{2}}. 
\end{align}
Thus, from (\ref{EE}) and the fact that $\varphi_{x}, \eta$ are uniformly bounded in $\mathnormal{L}^{2}(L_{0},L)$ in particular in $\mathnormal{L}^{2}(\beta_{1},\beta_{2}+3\delta)$ and $\left\|(f_{5})_{x} \right\|_{\mathnormal{L}^{2}(L_{0},L)}=o(1) $, we have the first estimation (\ref{the1}).\\
Next, from (\ref{fp7}) and (\ref{aap8}), we have
\begin{equation}\label{Hl1}
\mathrm{i}\lambda \psi_{x}-\xi_{x}= \lambda^{-\ell} \left( f_{7}\right)_{x} \;\; \;\; \text{and} \;\; \;\;\mathrm{i} \lambda \xi- \psi_{xx}-c_{2}(x)\eta=\lambda^{-\ell}f_{8}.
\end{equation}
Multiplying the above equations by $2g \overline{\xi}$ and $2g \overline{\psi_{x}}$ respectively, integrating over $ \left( \beta_{1},\beta_{2}+3\delta\right) $, taking the real part, then using the fact that $\psi_{x}, \xi$ are uniformly bounded in  $\mathnormal{L}^{2}(L_{0},L)$ in particular in $\mathnormal{L}^{2}(\beta_{1},\beta_{2}+3\delta)$, $\left\|f_{8} \right\|_{\mathnormal{L}^{2}(L_{0},L)}=o(1) $ and $\left\|\left( f_{7}\right) _{x} \right\|_{\mathnormal{L}^{2}(L_{0},L)}=o(1) $, we have 
\begin{equation}\label{E2}
\Re\left\lbrace 2 \mathrm{i} \lambda  \int_{\beta_{1}}^{\beta_{2}+3\delta}  g \psi_{x} \overline{\xi} dx\right\rbrace - \int_{\beta_{1}}^{\beta_{2}+3\delta}  g \left(  \left| \xi \right| ^{2}\right) _{x} dx= o\left(\lambda ^{-\ell} \right) ,
\end{equation}
and
\begin{align}\label{EE2}
\Re&\left\lbrace 2 \mathrm{i} \lambda  \int_{\beta_{1}}^{\beta_{2}+3\delta}  g \xi \overline{\psi_{x}} dx\right\rbrace - \int_{\beta_{1}}^{\beta_{2}+3\delta}  g\left(   \left| \psi_{x} \right| ^{2}\right) _{x} dx -
\Re\left\lbrace 2   \int_{\beta_{1}}^{\beta_{2}+3\delta}  c_{2}(x)g \eta \overline{\psi_{x}} dx\right\rbrace \\ \nonumber
&=o\left( \lambda ^{-\ell}\right) .
\end{align}
Adding (\ref{E2}) and (\ref{EE2}), then using integration by parts, we obtain
\begin{align}\label{E110}
& \left[ -g\left( \left| \xi \right| ^{2} + \left| \psi_{x} \right| ^{2}\right) \right] _{\beta_{1}}^{\beta_{2}+3\delta}= - \int_{\beta_{1}}^{\beta_{2}+3\delta}  g'  \left( \left| \xi \right| ^{2}+ \left| \psi_{x} \right| ^{2}\right) dx \\ \nonumber
&+ \Re\left\lbrace 2   \int_{\beta_{1}}^{\beta_{2}+3\delta}  c_{2}(x)g \eta \overline{\psi_{x}} dx\right\rbrace 
+  
o\left(  \lambda^{-\ell}\right) .
\end{align}
Using the definition of $g$ and Cauchy-Schwarz inequality in the above equation, we obtain
\begin{align}\label{E00}
&\left| \xi \left(\beta_{2}+3\delta \right) \right|^{2} +\left| \xi \left(\beta_{1} \right) \right|^{2} +\left| \psi_{x} \left(\beta_{2}+3\delta \right) \right|^{2} +\left| \psi_{x} \left(\beta_{1} \right) \right|^{2} \leq M_{g'} \int_{\beta_{1}}^{\beta_{2}+3\delta}\left(  \left| \xi\right|^{2}+ \left| \psi_{x}\right|^{2} \right) dx \\ \nonumber
&
+2 \left|  c_{2}\right|   M_{g} \left(\int_{\beta_{1}}^{\beta_{2}+3\delta} \left| \eta \right|^{2} dx \right) ^{\frac{1}{2}}
\left(\int_{\beta_{1}}^{\beta_{2}+3\delta} \left| \psi_{x}\right|^{2} dx \right) ^{\frac{1}{2}} 
+o \left( \lambda^{-\ell}  \right) .
\end{align}
Finally, from (\ref{E00}) and the fact that $\psi_{x}, \xi$ and $\eta$ are uniformly bounded in $\mathnormal{L}^{2}(L_{0},L)$ in particular in $\mathnormal{L}^{2}(\beta_{1},\beta_{2}+3\delta)$, we have the second estimation (\ref{the2}).	
\end{proof}
\begin{lemma}\label{lemme6}
	Let $h \in C^{1}\left(L_{0}, L\right) $ be a function with $h(L_{0}) = h(L) = 0$. The solution $(u, v, y, z,\varphi, \eta, \psi, \xi) \in  \mathnormal{D}(\mathcal{A})$ of (\ref{lpo})-(\ref{aap8}) satisfies the following asymptotic behavior estimation
	\begin{align}\label{lemm6}
	&\int_{L_{0}}^{L}  h'\left(  \left| \eta\right| ^{2} + a_{2} \left| \varphi_{x}\right| ^{2} + \left| \xi\right| ^{2} + \left| \psi_{x}\right| ^{2} \right) dx
	+\Re\left\lbrace 2  \int_{L_{0}}^{L} c_{2}(x) h \xi \overline{\varphi_{x}} dx\right\rbrace \\ \nonumber
	& +\Re\left\lbrace 2  \int_{L_{0}}^{L} d_{2}(x) h \eta \overline{\varphi_{x}} dx\right\rbrace -  \Re\left\lbrace 2  \int_{L_{0}}^{L} c_{2}(x) h \eta \overline{\psi_{x}} dx\right\rbrace = o\left( \lambda ^{- \ell}\right) .
	 \end{align}
\end{lemma}
\begin{proof}
Multiplying (\ref{ap6}) by $2h\overline{\varphi_{x}}$, integrating over $(L_{0}, L)$, taking the real part, then using the fact that $\varphi_{x}$ is uniformly bounded in $\mathnormal{L}^{2}(L_{0}, L)$ and $\left\| f_{6}\right\| _{\mathnormal{L}^{2}(L_{0}, L)} = o(1)$, we get	
\begin{align}\label{lati1}
&\Re\left\lbrace 2 \mathrm{i} \lambda  \int_{L_{0}}^{L}  h \eta \overline{\varphi_{x}} dx\right\rbrace - a_{2} \int_{L_{0}}^{L}  h \left( \left|\varphi_{x}\right| ^{2}\right)_{x} dx + \Re\left\lbrace 2   \int_{L_{0}}^{L} c_{2}(x) h \xi \overline{\varphi_{x}} dx\right\rbrace \\ \nonumber
&+  \Re\left\lbrace 2   \int_{L_{0}}^{L} d_{2}(x) h \eta \overline{\varphi_{x}} dx\right\rbrace = o\left( \lambda ^{- \ell}\right) .
\end{align}
From (\ref{aP}), we deduce that
$$
  \mathrm{i}\lambda\overline{\varphi_{x}}=  -\overline{\eta_{x}}-\lambda^{-\ell}\overline{\left( f_{5}\right) _{x}}.
$$	
Inserting the above equation in (\ref{lati1}), then using the fact that $\eta$ is uniformly bounded in $\mathnormal{L}^{2}(L_{0}, L)$ and $\left\| \left( f_{5}\right) _{x}\right\| _{\mathnormal{L}^{2}(L_{0}, L)} = o(1)$, we get		
\begin{align}\label{latif1}
&- \int_{L_{0}}^{L}  h \left( \left| \eta \right| ^{2} + a_{2} \left|\varphi_{x} \right| ^{2}\right)_{x} dx + \Re\left\lbrace 2  \int_{L_{0}}^{L} c_{2}(x) h \xi \overline{\varphi_{x}} dx\right\rbrace \\ \nonumber
& + \Re\left\lbrace 2  \int_{L_{0}}^{L} d_{2}(x) h \eta \overline{\varphi_{x}} dx\right\rbrace = o\left( \lambda ^{- \ell}\right) .
\end{align}	
Using integration by parts in (\ref{latif1}) and the fact that $h(L_{0}) = h(L)= 0$, we obtain
\begin{align}\label{latiff1}
&\int_{L_{0}}^{L}  h' \left( \left| \eta \right| ^{2} + a_{2} \left|\varphi_{x}\right| ^{2}\right) dx + \Re\left\lbrace 2   \int_{L_{0}}^{L} c_{2}(x) h \xi \overline{\varphi_{x}} dx\right\rbrace \\ \nonumber
& + \Re\left\lbrace 2   \int_{L_{0}}^{L} d_{2}(x) h \eta \overline{\varphi_{x}} dx\right\rbrace = o\left( \lambda ^{- \ell}\right) .
\end{align}		
Next, multiplying (\ref{aap8}) by $2h\overline{\psi_{x}}$, integrating over $(L_{0}, L)$, taking the real part, then using the fact that $\psi_{x}$ is uniformly bounded in $\mathnormal{L}^{2}(L_{0}, L)$ and $\left\| f_{8}\right\| _{\mathnormal{L}^{2}(L_{0}, L)} = o(1)$, we obtain	
\begin{equation}\label{lati}
\Re\left\lbrace 2 \mathrm{i} \lambda  \int_{L_{0}}^{L}  h \xi \overline{\psi_{x}} dx\right\rbrace - \int_{L_{0}}^{L}  h \left( \left|\psi_{x} \right| ^{2}\right)_{x} dx - \Re\left\lbrace 2  \int_{L_{0}}^{L} c_{2}(x) h \eta \overline{\psi_{x}} dx\right\rbrace  = o\left( \lambda ^{- \ell}\right) .
\end{equation}
from (\ref{fp7}), we deduce that
$$
\mathrm{i}\lambda\overline{\psi_{x}}=  -\overline{\xi_{x}}-\lambda^{-\ell}\overline{\left( f_{7}\right) _{x}}.
$$	
Inserting the above equation in (\ref{lati}), then using the fact that $\xi$ is uniformly bounded in $\mathnormal{L}^{2}(L_{0}, L)$ and $\left\| \left( f_{7}\right) _{x}\right\| _{\mathnormal{L}^{2}(L_{0}, L)} = o(1)$, we get	
\begin{equation}\label{latif}
- \int_{L_{0}}^{L}  h \left(\left| \xi \right| ^{2} + \left|\psi_{x}\right| ^{2}\right)_{x} dx - \Re\left\lbrace 2   \int_{L_{0}}^{L} c_{2}(x) h \eta \overline{\psi_{x}} dx\right\rbrace = o\left( \lambda ^{- \ell}\right) .
\end{equation}	
Using integration by parts in (\ref{latif}) and the fact that $h(L_{0}) = h(L)= 0$, we obtain
\begin{equation}\label{latiff}
\int_{L_{0}}^{L}  h' \left(\left| \xi \right| ^{2} + \left|\psi_{x} \right| ^{2}\right) dx - \Re\left\lbrace 2   \int_{L_{0}}^{L} c_{2}(x) h \eta \overline{\psi_{x}} dx\right\rbrace = o\left( \lambda ^{- \ell}\right) .
\end{equation}	
Finally, adding (\ref{latiff1}) and (\ref{latiff}), we obtain the desired estimation (\ref{lemm6}). 	
\end{proof}
Let  $\max\left( \beta_{2},\frac{\beta_{3}-\beta_{2}}{5}\right)<\delta< \beta_{4} $, we fix cut-off functions $\theta_{4},\theta_{5} \in C^{1}(L_{0},L)$ such that $0\leq \theta_{4}(x) \leq1$, $0\leq \theta_{5}(x) \leq1$ for all $x \in \left[L_{0},L\right]  $ and 
\begin{equation*}
\theta_{4}(x)=
\begin{cases}
1 & \text{ if} \;  x \in \left[ L_{0} ,\beta_{2}+3\delta\right],  
\\
0 &  \text{ if }\;    x \in \left[\beta_{3}-2\delta,L\right], 
\end{cases} 
\end{equation*}
and 
\begin{equation*}
\theta_{5}(x)=
\begin{cases}
0 & \text{ if} \;  x \in \left[ L_{0} ,\beta_{2}+3\delta\right],  
\\
1 &  \text{ if }\;    x \in \left[\beta_{3}-2\delta,L\right] ,
\end{cases} 
\end{equation*}
and set $\max_{\left[ L_{0},L\right] }\left| \theta'_{4}(x)\right| =M_{\theta'_{4}}$ and $\max_{\left[ L_{0},L\right] }\left| \theta'_{5}(x)\right| =M_{\theta'_{5}}$.
\begin{lemma}\label{lemme7}
Let  $\max\left( \beta_{2},\frac{\beta_{3}-\beta_{2}}{5}\right)<\delta< \beta_{4} $. The solution $(u, v, y, z,\varphi, \eta, \psi, \xi) \in  \mathnormal{D}(\mathcal{A})$ of (\ref{lpo})-(\ref{aap8}) satisfies the following asymptotic behavior estimations
\begin{equation} \label{esti1}
\int_{L_{0}}^{\beta_{2}+3\delta}  \left(  \left| \eta \right| ^{2} + a_{2} \left| \varphi_{x} \right| ^{2} + \left| \xi \right| ^{2} + \left| \psi_{x} \right| ^{2} \right) dx \leq  K_{1} \left|  a_{2}-1 \right|  \left|\lambda \right|  o\left(  \lambda  ^{-\min\left(\frac{\ell}{2},\frac{\ell}{4}+\frac{1}{2} \right) }\right)+o(1),
\end{equation}
and
\begin{equation} \label{esti2}
\int_{\beta_{3}-2\delta}^{L}  \left(  \left| \eta \right| ^{2} + a_{2} \left| \varphi_{x} \right| ^{2} + \left| \xi \right| ^{2} + \left| \psi_{x} \right| ^{2} \right) dx \leq  K_{2} \left|  a_{2}-1 \right|  \left|\lambda \right|  o\left(  \lambda  ^{-\min\left(\frac{\ell}{2},\frac{\ell}{4}+\frac{1}{2} \right) }\right)+o(1),
\end{equation}
where $K_{1}=\frac{4}{\left| c_{2}\right| }\left(1+\left( \beta_{3}-2\delta-L_{0}\right) M_{\theta'_{4}} \right) $ and $K_{2}=\frac{4}{\left| c_{2}\right| }\left(1+\left( L-\beta_{2}-3\delta \right) M_{\theta'_{5}} \right) $.
\end{lemma}
\begin{proof}
First, using the result of Lemma\ref{lemme6} with $h=\left( x-L_{0}\right) \theta_{4}$, we obtain
\begin{align} \label{Ca}\nonumber
&\int_{L_{0}}^{\beta_{2}+3\delta}  \left(  \left| \eta\right| ^{2} + a_{2} \left| \varphi_{x} \right| ^{2} + \left| \xi\right| ^{2} + \left| \psi_{x} \right| ^{2} \right) dx =\\ \nonumber
& -\int_{\beta_{2}+3\delta}^{\beta_{3}-2\delta} \left( \theta_{4}+\left( x-L_{0}\right) \theta'_{4}\right)  \left(  \left| \eta \right| ^{2} + a_{2} \left| \varphi_{x} \right| ^{2} + \left| \xi \right| ^{2} + \left| \psi_{x} \right| ^{2} \right) dx \\ 
&- \Re\left\lbrace 2c_{2} \int_{\beta_{2}+3\delta}^{\beta_{3}-2\delta} \left( x-L_{0}\right) \theta_{4} \xi \overline{\varphi_{x}} dx\right\rbrace  - \Re\left\lbrace 2d_{2} \int_{\beta_{2}}^{\beta_{3}-2\delta} \left( x-L_{0}\right) \theta_{4} \eta \overline{\varphi_{x}} dx\right\rbrace \\
\nonumber
&+  \Re\left\lbrace 2c_{2} \int_{\beta_{2}+3\delta}^{\beta_{3}-2\delta} \left( x-L_{0}\right) \theta_{4} \eta \overline{\psi_{x}} dx\right\rbrace   - \Re\left\lbrace 2c_{2} \int_{\beta_{1}}^{\beta_{2}+3\delta} \left( x-L_{0}\right)  \xi \overline{\varphi_{x}} dx\right\rbrace \\ \nonumber
&+  \Re\left\lbrace 2c_{2} \int_{\beta_{1}}^{\beta_{2}+3\delta} \left( x-L_{0}\right) \eta \overline{\psi_{x}} dx\right\rbrace +o\left(\lambda ^{-\ell} \right) .
\end{align}
Using Cauchy-Schwarz inequality, we have
\begin{align}\label{ca}\nonumber
&- \int_{\beta_{2}+3\delta}^{\beta_{3}-2\delta} \left( \theta_{4}+\left( x-L_{0}\right) \theta'_{4}\right)  \left(  \left| \eta \right| ^{2} + a_{2} \left| \varphi_{x} \right| ^{2} + \left| \xi\right| ^{2} + \left| \psi_{x}\right| ^{2} \right) dx \\ \nonumber
&- \Re\left\lbrace 2c_{2} \int_{\beta_{2}+3\delta}^{\beta_{3}-2\delta} \left( x-L_{0}\right) \theta_{4} \xi \overline{\varphi_{x}} dx\right\rbrace
  - \Re\left\lbrace 2d_{2} \int_{\beta_{2}}^{\beta_{3}-2\delta} \left( x-L_{0}\right) \theta_{4} \eta \overline{\varphi_{x}} dx\right\rbrace \\ \nonumber
  &+  \Re\left\lbrace 2c_{2} \int_{\beta_{2}+3\delta}^{\beta_{3}-2\delta} \left( x-L_{0}\right) \theta_{4} \eta \overline{\psi_{x}} dx\right\rbrace \leq \\  
&  \left( 1+\left(\beta_{3}-2\delta-L_{0} \right) M_{\theta'_{4}} \right) \int_{\beta_{2}+3\delta}^{\beta_{3}-2\delta}  \left(  \left| \eta\right| ^{2} + a_{2} \left| \varphi_{x}\right| ^{2} + \left| \xi\right| ^{2} + \left| \psi_{x} \right| ^{2} \right) dx \\ \nonumber
&+2c_{2}\left( \beta_{3}-2\delta -L_{0}\right) 
\left( \int_{\beta_{2}+3\delta}^{\beta_{3}-2\delta} \left| \xi \right| ^{2} dx \right) ^{\frac{1}{2}}
\left( \int_{\beta_{2}+3\delta}^{\beta_{3}-2\delta} \left| \varphi_{x} \right| ^{2} dx \right) ^{\frac{1}{2}} \\ \nonumber
& +2d_{2}\left( \beta_{3}-2\delta -L_{0} \right) 
\left( \int_{\beta_{2}}^{\beta_{3}-2\delta} \left| \eta \right| ^{2} dx \right) ^{\frac{1}{2}}
\left( \int_{\beta_{2}}^{\beta_{3}-2\delta} \left| \varphi_{x} \right| ^{2} dx \right) ^{\frac{1}{2}} \\ \nonumber
&+2c_{2}\left( \beta_{3}-2\delta -L_{0} \right) 
\left( \int_{\beta_{2}+3\delta}^{\beta_{3}-2\delta} \left| \eta \right| ^{2} dx \right) ^{\frac{1}{2}}
\left( \int_{\beta_{2}+3\delta}^{\beta_{3}-2\delta} \left| \psi_{x} \right| ^{2} dx \right) ^{\frac{1}{2}} .
\end{align}
Inserting (\ref{ca}) in (\ref{Ca}), using Lemmas\ref{lemme1}-\ref{lemme4} and the fact that $\psi_{x}, \xi$ are uniformly
bounded in $\mathnormal{L}^{2}(L_{0}, L)$, we obtain
\begin{align} \label{Bwa}
&\int_{L_{0}}^{\beta_{2}+3\delta}  \left(  \left| \eta\right| ^{2} + a_{2} \left| \varphi_{x}\right| ^{2} + \left| \xi \right| ^{2} + \left| \psi_{x} \right| ^{2} \right) dx \\ \nonumber
& \leq \frac{4}{\left| c_{2}\right| }\left(1+\left( \beta_{3}-2\delta-L_{0}\right) M_{\theta'_{4}} \right)  \left|  a_{2}-1 \right|  \left|\lambda \right|  o\left(  \lambda  ^{-\min\left(\frac{\ell}{2},\frac{\ell}{4}+\frac{1}{2} \right) }\right)+\mathcal{I}+o(1),
\end{align}
where
\begin{equation}\label{Bb}
\mathcal{I}:= \Re\left\lbrace 2c_{2} \int_{\beta_{1}}^{\beta_{2}+3\delta} \left( x-L_{0}\right) \eta \overline{\psi_{x}} dx\right\rbrace  -\Re\left\lbrace 2c_{2} \int_{\beta_{1}}^{\beta_{2}+3\delta} \left( x-L_{0}\right) \xi \overline{\varphi_{x}} dx\right\rbrace .
\end{equation}
From (\ref{aP}) and (\ref{fp7}), we have
$$
\overline{\varphi_{x}}= \mathrm{i}\lambda^{-1} \overline{\eta_{x}}+ \mathrm{i} \lambda^{-(\ell+1)} \overline{\left( f_{5}\right) _{x}} \;\; \;\; \text{and} \;\; \;\; \overline{\psi_{x}}= \mathrm{i}\lambda^{-1} \overline{\xi_{x}}+ \mathrm{i} \lambda^{-(\ell+1)} \overline{\left( f_{7}\right) _{x}}.
$$
Inserting the above equations in (\ref{Bb}), then using the fact that $\eta$ and $\xi$ are uniformly bounded in $\mathnormal{L}^{2}(L_{0}, L)$ and $\left\| \left( f_{5}\right) _{x}\right\| _{\mathnormal{L}^{2}(L_{0}, L)} = o(1)$, $\left\| \left( f_{7}\right) _{x}\right\| _{\mathnormal{L}^{2}(L_{0}, L)} = o(1)$, we obtain	
\begin{align}
\mathcal{I}&= \Re\left\lbrace 2c_{2} \mathrm{i} \lambda^{-1} \int_{\beta_{1}}^{\beta_{2}+3\delta} \left( x-L_{0}\right) \eta \overline{\xi_{x}} dx\right\rbrace  -\Re\left\lbrace 2c_{2} \mathrm{i} \lambda^{-1} \int_{\beta_{1}}^{\beta_{2}+3\delta} \left( x-L_{0}\right) \xi \overline{\eta_{x}} dx\right\rbrace \\ \nonumber
& + o(\lambda^{-(\ell+1)}).
\end{align}	
Using integration by parts to the second term in the above equation, we obtain
\begin{equation}\label{BBb}
\mathcal{I}= \Re\left\lbrace 2c_{2} \mathrm{i} \lambda^{-1} \int_{\beta_{1}}^{\beta_{2}+3\delta} \xi  \overline{\eta} dx\right\rbrace  -\Re\left\lbrace 2c_{2} \mathrm{i} \lambda^{-1} \left[  \left( x-L_{0}\right) \xi \overline{\eta} \right] _{\beta_{1}}^{\beta_{2}+3\delta}\right\rbrace .
\end{equation}
From Lemma\ref{lemme5}, we deduce that
\begin{equation}\label{d01}
\left|\eta\left( \beta_{2}+3\delta\right)  \right| =O\left(  \lambda  ^{\frac{1}{2}}\right),\; \left|\eta\left( \beta_{1}\right)  \right| =O\left(  \lambda ^{\frac{1}{2}}\right),
\end{equation}
\begin{equation} \label{d02}
\left|\xi\left( \beta_{2}+3\delta\right)  \right| =O\left( 1\right), \; \left|\xi\left( \beta_{1}\right)  \right| =O\left( 1\right).
\end{equation}
Using Cauchy-Schwarz inequality, (\ref{d01}), (\ref{d02}) and the fact that $\eta, \xi$ are uniformly bounded in $\mathnormal{L}^{2}(L_{0}, L)$, we obtain
\begin{equation}
 \Re\left\lbrace 2c_{2} \mathrm{i} \lambda^{-1} \int_{\beta_{1}}^{\beta_{2}+3\delta} \xi  \overline{\eta} dx\right\rbrace  =O\left(  \lambda  ^{-1}\right) =o(1),
\end{equation}
and
\begin{equation}
  -\Re\left\lbrace 2c_{2} \mathrm{i} \lambda^{-1} \left[  \left( x-L_{0}\right)\xi \overline{\eta} \right] _{\beta_{1}}^{\beta_{2}+3\delta}\right\rbrace =O\left( \lambda ^{-\frac{1}{2}}\right) =o(1).
\end{equation}
Inserting the above estimations in (\ref{BBb}), we get
$$
\mathcal{I}=o(1).
$$
Finally, from the above estimation and (\ref{Bwa}), we obtain the desired estimation (\ref{esti1}).
Next, using the result of Lemma\ref{lemme6} with $h=\left( x-L\right) \theta_{5}$, we obtain
\begin{align} \nonumber
&\int_{\beta_{3}-2\delta}^{L}  \left(  \left| \eta\right| ^{2} + a_{2} \left| \varphi_{x}\right| ^{2} + \left| \xi\right| ^{2} + \left| \psi_{x}\right| ^{2} \right) dx =  -\Re\left\lbrace 2c_{2} \int_{\beta_{2}+3\delta}^{\beta_{3}} \left( x-L\right) \theta_{5} \xi \overline{\varphi_{x}} dx\right\rbrace\\ \nonumber
& -\int_{\beta_{2}+3\delta}^{\beta_{3}-2\delta} \left( \theta_{5}+\left( x-L\right) \theta'_{5}\right)  \left(  \left|\eta\right| ^{2} + a_{2} \left| \varphi_{x} \right| ^{2} + \left| \xi\right| ^{2} + \left| \psi_{x} \right| ^{2} \right) dx \\  \nonumber
&  - \Re\left\lbrace 2d_{2} \int_{\beta_{2}+3\delta}^{\beta_{4}} \left( x-L\right) \theta_{5} \eta \overline{\varphi_{x}} dx\right\rbrace +  \Re\left\lbrace 2c_{2} \int_{\beta_{2}+3\delta}^{\beta_{3}} \left( x-L\right) \theta_{5} \eta \overline{\psi_{x}} dx\right\rbrace +o\left( \lambda ^{-\ell} \right) .
\end{align}
Using Cauchy-Schwarz inequality in the above equation, we get
\begin{align} \nonumber
&\int_{\beta_{3}-2\delta}^{L}  \left(  \left| \eta\right| ^{2} + a_{2} \left| \varphi_{x}\right| ^{2} + \left| \xi\right| ^{2} + \left| \psi_{x}\right| ^{2} \right) dx \leq\\ \nonumber
& \left( 1+\left(L-\beta_{2}-3\delta \right) M_{\theta'_{5}} \right) \int_{\beta_{2}+3\delta}^{\beta_{3}-2\delta}  \left(  \left| \eta\right| ^{2} + a_{2} \left| \varphi_{x}\right| ^{2} + \left| \xi\right| ^{2} + \left| \psi_{x}\right| ^{2} \right) dx \\
&+2c_{2}\left( L-\beta_{2}-3\delta \right) 
\left( \int_{\beta_{2}+3\delta}^{\beta_{3}} \left| \xi \right| ^{2} dx \right) ^{\frac{1}{2}}
\left( \int_{\beta_{2}+3\delta}^{\beta_{3}} \left| \varphi_{x} \right| ^{2} dx \right) ^{\frac{1}{2}} \\ \nonumber
& +2d_{2}\left( L-\beta_{2}-3\delta \right) 
\left( \int_{\beta_{2}+3\delta}^{\beta_{4}} \left| \eta \right| ^{2} dx \right) ^{\frac{1}{2}}
\left( \int_{\beta_{2}+3\delta}^{\beta_{4}} \left| \varphi_{x} \right| ^{2} dx \right) ^{\frac{1}{2}} \\ \nonumber
&+2c_{2}\left( L-\beta_{2}-3\delta \right) 
\left( \int_{\beta_{2}+3\delta}^{\beta_{3}} \left| \eta \right| ^{2} dx \right) ^{\frac{1}{2}}
\left( \int_{\beta_{2}+3\delta}^{\beta_{3}} \left| \psi_{x} \right| ^{2} dx \right) ^{\frac{1}{2}} +o\left(\lambda^{-\ell} \right)  .
\end{align}
Thus, from the above inequality, Lemmas\ref{lemme1}-\ref{lemme4} and the fact that $\varphi_{x}, \psi_{x}, \xi$ are uniformly
bounded in $\mathnormal{L}^{2}(L_{0}, L)$, we get (\ref{esti2}).
\end{proof}
\begin{lemma}\label{lemme8}
	Let  $\max\left( \beta_{2},\frac{\beta_{3}-\beta_{2}}{5}\right)<\delta< \beta_{4} $. The solution $(u, v, y, z,\varphi, \eta, \psi, \xi) \in  \mathnormal{D}(\mathcal{A})$ of (\ref{lpo})-(\ref{aap8}) satisfies the following estimate
	\begin{align}\label{ESt}
	  &\left| v\left(L_{0} \right) \right| ^{2} + a_{1} \left| u_{x} \left(L_{0} \right)\right| ^{2} + \left| z \left(L_{0} \right)\right| ^{2} + \left| y_{x} \left(L_{0} \right) \right| ^{2}  \leq  \\ \nonumber
	  &K_{3} \left|  a_{2}-1 \right|  \left|\lambda \right|  o\left(  \lambda  ^{-\min\left(\frac{\ell}{2},\frac{\ell}{4}+\frac{1}{2} \right) }\right)+o(1),
	\end{align}
	where $K_{3}=\frac{K_{1}}{\left(\beta_{1}-L_{0} \right) }$.
\end{lemma}
\begin{proof}
The proof is split into two steps.\\
\textbf{Step 1.} Letting $h_{1} \in C^{1}\left(\left[ L_{0},\beta_{1}\right]  \right)$, the following estimate is targeted to prove
\begin{align}\label{Ks}
&-\int_{L_{0}}^{\beta_{1}} h'_{1} \left(  \left| \eta\right| ^{2} + a_{2} \left| \varphi_{x} \right| ^{2} + \left| \xi \right| ^{2} + \left| \psi_{x} \right| ^{2} \right) dx \\ \nonumber
&+ h_{1}\left(\beta_{1} \right) \left( \left| \eta \left(\beta_{1}\right) \right| ^{2} + a_{2} \left| \varphi_{x} \left(\beta_{1} \right)\right| ^{2} + \left| \xi \left(\beta_{1} \right)\right| ^{2} + \left| \psi_{x} \left(\beta_{1} \right) \right| ^{2}\right)    \\ \nonumber
&- h_{1}\left(L_{0} \right) \left(  \left| \eta \left(L_{0} \right) \right| ^{2} + a_{2} \left| \varphi_{x} \left(L_{0} \right)\right| ^{2} + \left| \xi \left(L_{0} \right)\right| ^{2} + \left| \psi_{x} \left(L_{0} \right) \right| ^{2}\right)  = o\left(\lambda ^{-\ell} \right) .
\end{align}
First, multiplying (\ref{ap6}) and (\ref{aap8}) by $2h_{1}\overline{\varphi_{x}}$ and $2h_{1}\overline{\psi_{x}}$ respectively, integrating over $(L_{0}, \beta_{1})$, taking the real part, then using the fact that $\varphi_{x}$ and $\psi_{x}$ are uniformly bounded in $\mathnormal{L}^{2}(L_{0}, L)$ in particular in $\mathnormal{L}^{2}(L_{0}, \beta_{1})$ and $\left\| f_{6}\right\| _{\mathnormal{L}^{2}(L_{0}, L)} = o(1)$, $\left\| f_{8}\right\| _{\mathnormal{L}^{2}(L_{0}, L)} = o(1)$, we obtain	
\begin{equation}\label{Pre}
\Re\left\lbrace 2 \mathrm{i} \lambda  \int_{L_{0}}^{\beta_{1}}  h_{1} \eta \overline{\varphi_{x}} dx\right\rbrace - a_{2} \int_{L_{0}}^{\beta_{1}}  h_{1} \left( \left|\varphi_{x}\right| ^{2}\right)_{x} dx = o\left( \lambda ^{- \ell}\right) ,
\end{equation}
and
\begin{equation}\label{Pree}
\Re\left\lbrace 2 \mathrm{i} \lambda  \int_{L_{0}}^{\beta_{1}}  h_{1} \xi \overline{\psi_{x}} dx\right\rbrace -  \int_{L_{0}}^{\beta_{1}}  h_{1} \left( \left|\psi_{x}\right| ^{2}\right)_{x} dx = o\left( \lambda ^{- \ell}\right) .
\end{equation}
from (\ref{aP}) and (\ref{fp7}), we deduce that
$$
\mathrm{i}\lambda\overline{\varphi_{x}}=  -\overline{\eta_{x}}-\lambda^{-\ell}\overline{\left( f_{5}\right)_{x}} \;\; \;\; \text{and} \;\; \;\; \mathrm{i}\lambda\overline{\psi_{x}}=  -\overline{\xi_{x}}-\lambda^{-\ell}\overline{\left( f_{7}\right) _{x}}.
$$	
Inserting the above equations in (\ref{Pre}) and (\ref{Pree}), then using the fact that $\eta$ and $\xi$ are uniformly bounded in $\mathnormal{L}^{2}(L_{0}, L)$ in particular in $\mathnormal{L}^{2}(L_{0}, \beta_{1})$ and $\left\| \left( f_{5}\right)_{x}\right\| _{\mathnormal{L}^{2}(L_{0}, L)} = o(1)$, $\left\| \left( f_{7}\right) _{x}\right\| _{\mathnormal{L}^{2}(L_{0}, L)} = o(1)$, we get		
\begin{equation}\label{ben1}
-\int_{L_{0}}^{\beta_{1}} h_{1} \left(  \left| \eta\right| ^{2} + a_{2} \left| \varphi_{x} \right| ^{2} \right) _{x} dx=o\left(  \lambda ^{-\ell} \right) ,
\end{equation}
and
\begin{equation}\label{ben2}
-\int_{L_{0}}^{\beta_{1}} h_{1} \left(  \left| \xi\right| ^{2} +  \left| \psi_{x} \right| ^{2} \right) _{x} dx=o\left( \lambda ^{-\ell} \right) .
\end{equation}
Adding (\ref{ben1}) and (\ref{ben2}), then using integration by parts, we get (\ref{Ks}).\\
\textbf{Step2.} Taking $h_{1}(x)=\left( x-\beta_{1}\right) $ in (\ref{Ks}) and using (\ref{esti1}), we obtain 
\begin{align}
&\left| \eta \left(L_{0} \right) \right| ^{2} + a_{2} \left| \varphi_{x}  \left(L_{0} \right)\right| ^{2} + \left| \xi \left(L_{0} \right)\right| ^{2} + \left| \psi_{x}  \left(L_{0} \right) \right| ^{2}  \leq  \\ \nonumber
&\frac{K_{1}}{\left( \beta_{1}-L_{0}\right) } \left|  a_{2}-1 \right|  \left|\lambda \right|  o\left(  \lambda  ^{-\min\left(\frac{\ell}{2},\frac{\ell}{4}+\frac{1}{2} \right) }\right)+o(1).
\end{align}
Using (\ref{aP}), (\ref{fp7}) and the transmission conditions (\ref{tran}), we get (\ref{ESt}).  
\end{proof}
\begin{lemma}\label{lemme9}
	Let  $\max\left( \beta_{2},\frac{\beta_{3}-\beta_{2}}{5}\right)<\delta< \beta_{4} $. The solution $(u, v, y, z,\varphi, \eta, \psi, \xi) \in  \mathnormal{D}(\mathcal{A})$ of (\ref{lpo})-(\ref{aap8}) satisfies the following estimate
	\begin{align}\label{Oc}
	& \int_{0}^{L_{0}}\left(  \left| v \right| ^{2} + a_{1} \left| u_{x} \right| ^{2} + \left| z \right| ^{2} + \left| y_{x} \right| ^{2}  \right) dx \leq  \\ \nonumber
	&L_{0} K_{3} \left|  a_{2}-1 \right|  \left|\lambda \right|  o\left(  \lambda  ^{-\min\left(\frac{\ell}{2},\frac{\ell}{4}+\frac{1}{2} \right) }\right)+o(1).
	\end{align}
\end{lemma}
\begin{proof}
	First, using the multipliers $2 x \overline{u_{x}}$ and $2 x \overline{y_{x}}$ for (\ref{ap}) and (\ref{aap}) respectively, integrating over $( 0,L_{0})$, taking the real part, then using the fact that $u_{x}$ and $y_{x}$ are uniformly bounded in $\mathnormal{L}^{2}(0,L_{0})$ and $\left\| f_{2}\right\| _{\mathnormal{L}^{2}(0,L_{0})} = o(1)$, $\left\| f_{4}\right\| _{\mathnormal{L}^{2}(0,L_{0})} = o(1)$, we obtain	
	\begin{equation}\label{Co}
	2\Re\left\lbrace  \mathrm{i} \lambda  \int_{0}^{L_{0}} x   v \overline{u_{x}} dx\right\rbrace - a_{1} \int_{0}^{L_{0}} x  \left( \left|u_{x}\right| ^{2}\right)_{x} dx + 2\Re\left\lbrace  \int_{0}^{L_{0}} x   c_{1}(x) y \overline{u_{x}} dx\right\rbrace = o\left( \lambda ^{- \ell}\right) ,
	\end{equation}
	and
	\begin{equation}\label{Coo}
   2\Re\left\lbrace  \mathrm{i} \lambda  \int_{0}^{L_{0}} x z \overline{y_{x}} dx\right\rbrace -  \int_{0}^{L_{0}} x  \left( \left|y_{x}\right| ^{2}\right)_{x} dx + 2\Re\left\lbrace  \int_{0}^{L_{0}} x   c_{1}(x) u \overline{y_{x}} dx\right\rbrace  = o\left( \lambda ^{- \ell}\right) .
	\end{equation}
	Using Cauchy-Schwarz inequality, the fact that $ \lambda u$, $ \lambda y$, $u_{x}$, $y_{x}$ are uniformly bounded in $\mathnormal{L}^{2}(0,L_{0})$, we get
	\begin{equation}\label{mw}
	 2\Re\left\lbrace  \int_{0}^{L_{0}} x   c_{1}(x) y \overline{u_{x}} dx\right\rbrace = O(\lambda^{-1})=o(1) \;\; \text{and} \;\;
	  2\Re\left\lbrace  \int_{0}^{L_{0}} x   c_{1}(x) u \overline{y_{x}} dx\right\rbrace= O(\lambda^{-1})=o(1).  
	\end{equation}
	Inserting (\ref{mw}) in (\ref{Co}) and (\ref{Coo}), we obtain
	\begin{equation}\label{Co'}
	2\Re\left\lbrace  \mathrm{i} \lambda  \int_{0}^{L_{0}} x   v \overline{u_{x}} dx\right\rbrace - a_{1} \int_{0}^{L_{0}} x  \left( \left|u_{x}\right| ^{2}\right)_{x} dx= o\left( 1\right) ,
	\end{equation}
	and
	\begin{equation}\label{Coo'}
	2\Re\left\lbrace  \mathrm{i} \lambda  \int_{0}^{L_{0}} x z \overline{y_{x}} dx\right\rbrace -  \int_{0}^{L_{0}} x  \left( \left|y_{x}\right| ^{2}\right)_{x} dx   = o\left( 1\right) .
	\end{equation}
	from (\ref{lpo}) and (\ref{fp}), we deduce that
	$$
	\mathrm{i}\lambda\overline{u_{x}}=  
	-\overline{v_{x}}-\lambda^{-\ell}\overline{\left( f_{1}\right) _{x}} \;\; \;\; \text{and} \;\; \;\; \mathrm{i}\lambda\overline{y_{x}}=  -\overline{z_{x}}-\lambda^{-\ell}\overline{\left( f_{3}\right)_{x}}.
	$$	
	Inserting the above equations in (\ref{Co'}) and (\ref{Coo'}), then using the fact that $v$ and $z$ are uniformly bounded in $\mathnormal{L}^{2}(0,L_{0})$ and $\left\| \left( f_{1}\right) _{x}\right\| _{\mathnormal{L}^{2}(0,L_{0})} = o(1)$, $\left\|\left(  f_{3}\right) _{x}\right\| _{\mathnormal{L}^{2}(0,L_{0})} = o(1)$, we get		
	\begin{equation}\label{Co1}
	-\int_{0}^{L_{0}} x \left(  \left| v \right| ^{2} + a_{1} \left| u_{x} \right| ^{2} \right) _{x} dx=o\left( 1 \right) ,
	\end{equation}
	and
	\begin{equation}\label{Coo1}
	-\int_{0}^{L_{0}} x  \left(  \left| z \right| ^{2} +  \left| y_{x} \right| ^{2} \right) _{x} dx=o\left( 1 \right)  .
	\end{equation}
	Adding (\ref{Co1}) and (\ref{Coo1}), then using integration by parts, we obtain 
	\begin{align}
	&\int_{0}^{L_{0}}\left( \left| v \right| ^{2} + a_{1} \left| u_{x} \right| ^{2} + \left| z \right| ^{2} + \left| y_{x}  \right| ^{2} \right) dx   \\ \nonumber
	&= L_{0}\left(  \left| v \left(L_{0} \right) \right| ^{2} + a_{1} \left| u_{x} \left(L_{0} \right)\right| ^{2} + \left| z \left(L_{0} \right)\right| ^{2} + \left| y_{x} \left(L_{0} \right) \right| ^{2}  \right) + o\left( 1 \right).
	\end{align}
	Using the above result and (\ref{ESt}), we get (\ref{Oc}). 
\end{proof}

\textbf{Proof of Theorem\ref{port}}. The proof of Theorem\ref{port} is divided into three steps.\\

 \textbf{Step 1.} By taking $a_{2}= 1$ and $\ell = 0$ in Lemmas\ref{lemme1}-\ref{lemme4}, we obtain
\begin{align}
&\int_{\beta_{2}}^{\beta_{4}} \left| \eta \right| ^{2} dx= o\left( 1\right) , \; \int_{\beta_{2}+\delta}^{\beta_{4}-\delta} \left| \varphi_{x} \right| ^{2} dx= o\left(1\right), \\ \nonumber
&\int_{\beta_{2}+2\delta}^{\beta_{3}-\delta} \left| \psi_{x} \right| ^{2} dx= o\left( 1\right) \;\; \text{and} \;\;  \int_{\beta_{2}+3\delta}^{\beta_{3}-2\delta} \left| \xi \right| ^{2} dx= o\left(1\right).
\end{align}
Consequently, we have 
\begin{equation}
\int_{\beta_{2}+3\delta}^{\beta_{3}-2\delta} \left( \left| \eta \right| ^{2} +\left| \varphi_{x} \right| ^{2} +\left| \xi \right| ^{2} +\left| \psi_{x} \right| ^{2} \right) dx= o\left(1\right).
\end{equation}
 \textbf{Step 2.} Using the fact that $a_{2}=1$ in Lemmas\ref{lemme7} and \ref{lemme9}, we obtain 
\begin{equation}
\int_{L_{0}}^{\beta_{2}+3\delta} \left( \left| \eta \right| ^{2} +\left| \varphi_{x} \right| ^{2} +\left| \xi \right| ^{2} +\left| \psi_{x} \right| ^{2} \right) dx= o\left(1\right),
\end{equation}
\begin{equation}
\int_{\beta_{3}-2\delta}^{L} \left( \left| \eta \right| ^{2} +\left| \varphi_{x} \right| ^{2} +\left| \xi \right| ^{2} +\left| \psi_{x} \right| ^{2} \right) dx= o\left(1\right),
\end{equation}
and
\begin{equation}
\int_{0}^{L_{0}} \left( \left| v \right| ^{2} +a_{1}\left|u_{x} \right| ^{2} +\left| z \right| ^{2} +\left| y_{x} \right| ^{2} \right) dx= o\left(1\right).
\end{equation}
 \textbf{Step 3.} According to \textbf{Step 1} and \textbf{Step 2}, we obtain $\left\| U\right\| _{\mathcal{H}}=o(1)$, which contradicts (\ref{hn}). This implies that
$$
	\sup_{\lambda \in \mathbb{R} } \left\|\left( \mathrm{i} \lambda I-\mathcal{A}\right) ^{-1} \right\| _{\mathcal{L}(\mathcal{H})}=O(1).
$$
So by Theorem\ref{expo}, we deduce that system (\ref{pb})-(\ref{init}) is exponentially stable.\\

\textbf{Proof of Theorem\ref{portt}}. The proof of Theorem\ref{portt} is divided into three steps.\\

\textbf{Step 1.} Taking $a_{2} \neq 1$, then from Lemmas\ref{lemme3} and \ref{lemme4}, we get 
  \begin{equation}
\int_{\beta_{2}+2\delta}^{\beta_{3}-\delta} \left| \psi_{x} \right| ^{2} dx= o\left(  \lambda  ^{-\min\left(\frac{\ell}{2},\frac{\ell}{4}+\frac{1}{2} \right) +1}\right) \;\; \text{and} \;\;  \int_{\beta_{2}+3\delta}^{\beta_{3}-2\delta} \left| \xi \right| ^{2} dx= o\left(  \lambda  ^{-\min\left(\frac{\ell}{2},\frac{\ell}{4}+\frac{1}{2} \right) +1}\right).
\end{equation}
Taking $\ell=2$ in the above estimations, we obtain 
\begin{equation}
\int_{\beta_{2}+2\delta}^{\beta_{3}-\delta} \left| \psi_{x} \right| ^{2} dx= o\left( 1\right) \;\; \text{and} \;\;  \int_{\beta_{2}+3\delta}^{\beta_{3}-2\delta} \left| \xi \right| ^{2} dx= o\left( 1\right).
\end{equation}
Taking $\ell=2$ in Lemma\ref{lemme1} and \ref{lemme2}, we get
\begin{equation}\label{77}
\int_{\beta_{2}}^{\beta_{4}} \left| \eta \right| ^{2} dx= o\left( \lambda^{-2}\right) \;\; \text{and} \;\;  \int_{\beta_{2}+\delta}^{\beta_{4}-\delta} \left| \varphi_{x} \right| ^{2} dx= o\left( \lambda^{-2}\right) .
\end{equation} 
In particular, we have
\begin{equation}
\int_{\beta_{2}+3\delta}^{\beta_{3}-2\delta} \left( \left| \eta \right| ^{2} +a_{2}\left| \varphi_{x} \right| ^{2} +\left| \xi \right| ^{2} +\left| \psi_{x} \right| ^{2} \right) dx= o\left(1\right).
\end{equation}
\textbf{Step 2.} Using the fact that $a_{2} \neq 1$ and $\ell=2$, then from Lemmas\ref{lemme7} and \ref{lemme9}, we obtain 
\begin{equation}
\int_{L_{0}}^{\beta_{2}+3\delta} \left( \left| \eta \right| ^{2} +a_{2}\left| \varphi_{x} \right| ^{2} +\left| \xi \right| ^{2} +\left| \psi_{x} \right| ^{2} \right) dx= o\left(1\right),
\end{equation}
\begin{equation}
\int_{\beta_{3}-2\delta}^{L} \left( \left| \eta \right| ^{2} +a_{2}\left| \varphi_{x} \right| ^{2} +\left| \xi \right| ^{2} +\left| \psi_{x} \right| ^{2} \right) dx= o\left(1\right),
\end{equation}
and
\begin{equation}
\int_{0}^{L_{0}} \left( \left| v \right| ^{2} +a_{1}\left|u_{x} \right| ^{2} +\left| z \right| ^{2} +\left| y_{x} \right| ^{2} \right) dx= o\left(1\right).
\end{equation}
\textbf{Step 3.} According to \textbf{Step 1} and \textbf{Step 2}, we obtain $\left\| U\right\| _{\mathcal{H}}=o(1)$, which contradicts (\ref{hn}). This implies that
$$
\sup_{\lambda \in \mathbb{R} } \left\|\left( \mathrm{i} \lambda I-\mathcal{A}\right) ^{-1} \right\| _{\mathcal{L}(\mathcal{H})}=O(\lambda^{2}).
$$
So by Theorem\ref{poly}, we deduce that system (\ref{pb})-(\ref{init}) is polynomially stable.

\section{Conclusion and future works}
We have studied the stabilization of a locally transmission problems of two wave systems. We proved the strong stability of the system by using Arendt and Batty criteria. We established the exponential stability of the solution if and only if the waves of the second coupled equations have the same speed propagation (i.e., $a_{2} = 1$). In the case $a_{2} \neq 1$, we proved that the energy of our problem decays polynomially with the rate $t^{-1}$. Finally, we present some open problems:
\begin{enumerate}
	\item Prove that the energy decay rate $t^{-1}$ is optimal.
	\item Study system (\ref{pb})-(\ref{init}) in the multidimensional case. 
	\item Can we get stability results if  $d_{1}(x) \neq 0\; \text{in} \; (0,L_{0})$.
\end{enumerate}

\appendix
\section{}
\label{appendix}
In order to make this paper more self-contained, we recall in this short appendix some notions and stability results used in this work.
\begin{definition}
	Assume that $\mathnormal{A}$ is the generator of a C$_{0}$-semigroup of contractions $\left( e^{t \mathnormal{A}}\right)_{t\geq 0} $ on a Hilbert space $\mathcal{H}$. The C$_{0}$-semigroup $\left( e^{t \mathnormal{A}}\right)_{t\geq 0} $	is said to be
	\begin{enumerate}
		\item strongly stable if
		$$
		\lim_{t \rightarrow +\infty} \left\|e^{t \mathnormal{A}}x_{0} \right\|_{\mathcal{H}} =0, \;\; \;\; \forall \; x_{0} \in \mathcal{H}; 
		$$
		\item exponentially (or uniformly) stable if there exist two positive constants $M$ and $\epsilon$ such	that
		$$
		\left\|e^{t \mathnormal{A}}x_{0} \right\|_{\mathcal{H}} \leq M e^{-\epsilon t} \left\| x_{0}\right\|_{\mathcal{H}}, \;\; \;\; \forall \; t>0, \; \forall \; x_{0} \in \mathcal{H};   
		$$
		\item polynomially stable if there exists two positive constants $C$ and $\alpha$ such that
		$$
		\left\|e^{t \mathnormal{A}}x_{0} \right\|_{\mathcal{H}} \leq C t^{-\alpha} \left\| x_{0}\right\|_{\mathcal{H}}, \;\; \;\; \forall \; t>0, \; \forall \; x_{0} \in \mathcal{H} .   
		$$
	\end{enumerate}
\end{definition}

To show the strong stability of a C$_{0}$-semigroup of contraction $\left( e^{t \mathnormal{A}}\right)_{t\geq 0} $ we rely on the
following result due to Arendt-Batty \cite{AW}.
\begin{theorem}\label{fort}
	Assume that $\mathnormal{A}$ is the generator of a C$_{0}$-semigroup of contractions $\left( e^{t \mathnormal{A}}\right)_{t\geq 0} $ on a Hilbert space $\mathcal{H}$. If
	\begin{enumerate}
		\item $\mathnormal{A}$ has no pure imaginary eigenvalues,
		\item $\sigma (\mathnormal{A} ) \cap \mathrm{i}\mathbb{R} $ is countable,
	\end{enumerate} 
	where $\sigma (\mathnormal{A} )$ denotes the spectrum of $\mathnormal{A}$, then the C$_{0}$-semigroup $\left( e^{t \mathnormal{A}}\right)_{t\geq 0} $ is strongly stable.
\end{theorem}
Concerning the characterization of exponential stability of a C$_{0}$-semigroup of contraction
$\left( e^{t \mathnormal{A}}\right)_{t\geq 0} $ we rely on the following result due to Huang \cite{hu} and Pr$\ddot{u}$ss \cite{pr}.
\begin{theorem}\label{expo}
	Let $\mathnormal{A} : \mathnormal{D}\left( \mathnormal{A} \right) \subset \mathcal{H}\rightarrow  \mathcal{H} $ generate a C$_{0}$-semigroup of contractions $\left( e^{t \mathnormal{A}}\right)_{t\geq 0} $ on $\mathcal{H}$. Assume that $\mathrm{i}\lambda \in \rho(\mathnormal{A})$, $\forall \lambda \in \mathbb{R}$. Then, the C$_{0}$-semigroup $\left( e^{t \mathnormal{A}}\right)_{t\geq 0} $ is exponentially stable if and only if	
	$$
	\sup_{\lambda \in \mathbb{R} } \left\|\left( \mathrm{i} \lambda I-\mathnormal{A}\right) ^{-1} \right\| _{\mathcal{L}(\mathcal{H})}<+\infty.
	$$
\end{theorem}

Also, concerning the characterization of polynomial stability of a C$_{0}$-semigroup of contraction $\left( e^{t \mathnormal{A}}\right)_{t\geq 0} $ we rely on the following result due to Borichev and Tomilov \cite{BT}.
\begin{theorem}\label{poly}
	Assume that $\mathnormal{A}$ is the generator of a strongly continuous semigroup of contractions $\left( e^{t \mathnormal{A}}\right)_{t\geq 0} $ on $\mathcal{H}$. If $\mathrm{i} \mathbb{R} \subset \rho (\mathnormal{A})$, then for a fixed $\ell>0$ the following conditions are equivalent:
	\begin{equation}\label{A.39} 
	\sup_{\lambda \in \mathbb{R} } \left\|\left( \mathrm{i} \lambda I-\mathnormal{A}\right) ^{-1} \right\| _{\mathcal{L}(\mathcal{H})}=O\left(\left| \lambda \right|^{\ell}  \right) ,
	\end{equation}
	
	\begin{equation}
	\left\|e^{t \mathnormal{A}} U_{0}\right\|^{2}_{\mathcal{H}} \leq \frac{C}{t^{\frac{2}{\ell}}} \left\|U_{0} \right\|^{2}_{\mathnormal{D}(\mathnormal{A})}, \;\; \forall t>0, \; U_{0} \in {\mathnormal{D}(\mathnormal{A})}, \; \text{for some} \; C>0.  
	\end{equation}
\end{theorem}

\end{document}